\documentclass[12pt,a4paper,twoside]{amsart}

\usepackage{amsmath,amssymb}
\usepackage{tikz,pgf}

\theoremstyle{plain}
\newtheorem{lemma}{Lemma}[section]
\newtheorem*{lemma*}{Lemma}

\newtheorem{theorem}{Theorem}[section]
\newtheorem{corollary}{Corollary}[section]

\theoremstyle{definition}
\newtheorem{definition}{Definition}[section]

\newtheorem*{notation*}{Notation}

\newtheorem{example}{Example}[section]
\newtheorem{assumption}{Assumption}[section]

\newcommand{\Z}{\mathbb{Z}}

\newcommand{\R}{\mathbb{R}}
\newcommand{\cl}{\ensuremath{\mathrm{cl}\,}}

\DeclareMathOperator*{\Res}{Res}

\newcommand{\sgn}{\mathrm{sgn\,}}

\newcommand{\floor}[1]{\ensuremath{\left\lfloor #1 \right\rfloor}}

\renewcommand{\vec}[1]{ \ensuremath{ {\bf#1}} }
\newcommand{\selberg}[1]{\ensuremath{\left\langle #1 \right\rangle}}

\newcommand{\Sls}{\ensuremath{S_{\textup{ls}}}}
\newcommand{\Srt}{\ensuremath{S_{\textup{rt}}}}

\title[Asymptotics for type $A$ zero drift reflectable walks]{Asymptotics for the number of zero drift reflectable walks in a Weyl chamber of type $A$}
 \author[Thomas Feierl]{Thomas Feierl$^\ddagger$}
\thanks{$^\ddagger$ Fakult\"at f\"ur Mathematik, Universit\"at Wien}
\address{Thomas Feierl \\ Fakult\"at f\"ur Mathematik \\ Universit\"at Wien}
\date{\today}

\begin{document}
\maketitle

\begin{abstract}
We study lattice walks in a Weyl chamber of type $A$ with fixed or free end points.
For lattice walk models with zero drift that may be counted by means of a reflection argument, we determine asymptotics for the number of such walks as their length tends to infinity.
These models are equivalent to the lock step model and the random turns model of vicious walkers.
As special cases, our main results include various asymptotic formulas found in the literature.
\end{abstract}

\section{Introduction}

We consider lattice walks confined to the region $x_1<\dots<x_k$, where $x_j$ refers to the $j$-th coordinate in $\R^k$.
This region (or cone) may be identified with a Weyl chamber of a reflection group of type $A$.
The steps the walks may consist of are chosen in a way such that resulting walks are reflectable (see definition~\ref{def:reflectable walks}).
Additionally, we will require the walks to have zero drift.

Such lattice walk models are natural objects of study in combinatorics and probability theory, and, e.g., 
also play a role in statistical physics, where they serve as models for (fermionic) particle models.
In particular, the models considered here are related to two vicious walkers models, namely the \emph{random turns model} and the \emph{lock step model}.
As the number of steps the walks consist of tends to infinity, the lattice path model (properly scaled) converges to non-colliding Brownian motion, which is the eigenvalue process of certain ensembles of random matrices (see Mehta~\cite{MR2129906}).

This paper has to be understood as a continuation of the studies started in \cite{RSA:RSA20467}.
There, the author determined asymptotics for the number of reflectable random walks confined to the region $0<x_1<\dots<x_k$, which may be interpreted as a Weyl chamber of type $B$.
(It is interesting to note that the additional positivity constraint automatically ensures that admissible models satisfy the ``zero drift'' requirement.)

In this manuscript, we determine asymptotics for the number of walks with either a fixed end point or with a free end point as the number of steps tends to infinity.
As main results we provide first and second order asymptotics for these quantities.
In principle, though, one could determine asymptotics of arbitrary order using the techniques applied in this paper.
The derivation of the results essentially relies on the application of the saddle point method, though there are some obstacles to overcome.
First, the analysis requires the information on the asymptotic behaviour of functions defined through determinants, a problem put to the forth by Tate and Zelditch~\cite{MR2102573} (see also the discussion in the introduction of \cite{RSA:RSA20467}).
In our case we not only determine the dominant asymptotic behaviour of these determinantal functions, but provide the complete Taylor series expansion.
This, as already mentioned, allows us in principle to determine asymptotics of arbitray order.
Second, in the case of a free end point, the asymptotics depend on the parity of the dimension $k$, which requires a technically demanding transformation before the application of the saddle point method (see lemma~\ref{lem:free_end_point_integral_reduction}).
As special cases, our results contain asymptotic formulae obtained by Krattenthaler et al and Rubey~\cite{rubey}.

This paper is organised as follows.
In the next section, we describe the lattice path model underlying this manuscript, state some fundamental results, and finally state this paper's main results.
In section~\ref{sec:asymptotics and determinants}, we derive complete Taylor series expansions for a class of functions defined via determinants which are required for our asymptotic analysis.
Proofs of our main results are given in sections~\ref{sec:proofs:fixed} and \ref{sec:proofs:free}.

\section{Main results}
In this section, we give a description of the lattice path model considered, and, at the end of this section, state the main results of this paper.
\subsection{Reflectable walks}
We start with a detailed description of the lattice walk model underlying this manuscript and fixing the basic notation.
By $\mathcal S\subset\R^k$ we denote the step set.
This set is assumed to be finite.
The lattice $\mathcal L$ our walks ``live on'' is the $\Z$-lattice spanned by $\mathcal S$.
Now, a lattice walk of length $n$ from $\vec u\in\mathcal L$ to $\vec v\in\mathcal L$ is a sequence $(\vec s^{(1)},\dots,\vec s^{(n)})$ of $n$ elements of $\mathcal S$ such that $\vec u+\sum_{j=1}^{n}\vec s^{(j)}=\vec v$.

Here, we are interested in those lattice walks that stay inside the region
\[ \mathcal W=\left\{ (x_1,\dots,x_k)\in\R^k\ :\ x_1<x_2<\dots<x_k\right\}. \]
The reader should observe that this region's boundary is contained in the union of all hyperplanes defined by the equations
\begin{equation}
	x_{j+1}-x_j=0,\qquad 1\le j< k.
\label{eq:hyperplanes type A}
\end{equation}
The set of reflections $r_j$, $1\le j<k$, associated with these hyperplanes forms a minimal generator for the reflection group of type $A_{k-1}$.
The hyperplanes associated with all reflections of the group $A_{k-1}$ divide the $\R^k$ into connected regions which are called \emph{Weyl chambers}.
One of these Weyl chambers is the region $\mathcal W$ our walks are confined to.
This should explain the title of this manuscript (the ``zero drift'' in the title means that $\vec s\in\mathcal S$ if and only if $-\vec s\in\mathcal S$, but we come to that later on).
Observe that, on the coordinate level, the reflection $r_j$ simply interchanges the coordinates $j$ and $j+1$.
Thus, $A_{k-1}$ is isomporphic to $\mathfrak S_k$, the symmetric group on $\left\{ 1,2,\dots,k \right\}$.
Under this isomporhpism, the reflection $r_j$ corresponds to the transposition $(j,j+1)$.

\begin{definition}
	The step set $\mathcal S\subset\R^k$ defines a \emph{reflectable walk model of type $A_{k-1}$} if and only if
	\begin{enumerate}
		\item $\vec s\in\mathcal S$ implies $r(\vec s)\in\mathcal S$ for each reflection $r\in A_{k-1}$ and
		\item for each $s\in\mathcal S$ and each $\vec u\in\mathcal L\setminus\cl\mathcal W$ we have $\vec u+\vec s\in\cl\mathcal W$,
	\end{enumerate}
	where $\cl\mathcal W$ denotes the closure of $\mathcal W$, i.e., the set $x_1\le x_2\le\dots\le x_k$.
\label{def:reflectable walks}
\end{definition}
Thus, a step set corresponding to a reflectable walk model of type $A_{k-1}$ is invariant under the action of the symmetric group on the coordinates.
Further, it is impossible for the walker to enter the region $\mathcal W$ from the outside without first landing at the border of $\mathcal W$ and vice versa.

In this manuscript, we consider only lattice walks with zero drift.
Therefore, we make the following assumption.
\begin{assumption}
	In this manuscript, we assume the random walk has zero drift, i.e.,
	\[  \sum_{\vec s\in\mathcal S}\vec s = \vec 0. \]
	\label{assum:zero drift}
\end{assumption}

Now, for any $\vec u,\vec v\in\mathcal L$, we denote by $P_n(\vec u\to\vec v)$ the cardinality of the set of $n$-step walks from $\vec u$ to $\vec v$, i.e.,
\[  P_n(\vec u\to\vec v) = \left|\left\{ (\vec s^{(1)},\dots,\vec s^{(n)})\in\mathcal S^n\ :\ \vec u+\sum_{j=1}^n\vec s^{(j)}=\vec v \right\}\right|. \]
Our primary focus lies on those walks that stay within $\mathcal W$ all of the time, i.e., in the set
\[ P_n^+(\vec u\to\vec v) = \left|\left\{ (\vec s^{(1)},\dots,\vec s^{(n)})\in S^n\ :\ \ \vec u+\sum_{j=1}^n\vec s^{(j)}=\vec v\ \textrm{and}\ \vec u+\sum_{j=1}^q\vec s^{(j)}\in\mathcal W,\ 0\le q\le n\right\}\right|. \]
Clearly, if not $\vec u,\vec v\in\mathcal L\cap\mathcal W$ then we have $P_n^+(\vec u\to\vec v)=0$.

The requirements of definition~\ref{def:reflectable walks} are necessary and sufficient conditions for proving the following theorem using an elegant generalisation of Andre's reflection principle.
For the sake of simplicity we prove the lemma for the type $A_{k-1}$ case, the only case relevant to this manuscript.
It should be stressed however that the following result still holds true even with the group $A_{k-1}$ replaced by any finite or affine reflection group.
For details, we refer the reader directly to Gessel and Zeilberger's paper~\cite{MR1092920}.
\begin{lemma}[{see Gessel and Zeilberger~\cite[Theorem 1]{MR1092920}}]
	Let $\mathcal S$ satisfy the requirements of definition~\ref{def:reflectable walks}.
	Then, for any natural number $n$ and any two lattice points $\vec u,\vec v\in\mathcal L\cap\mathcal W$ we have
	\begin{equation}
		P_n^+(\vec u\to\vec v) = \sum_{r\in A_{k-1}}(-1)^{l(r)}P_n(r(\vec u)\to\vec v),
	\label{eq:reflection principle}
	\end{equation}
	where $l(r)$ denotes the minimum possible number of factors required to express $r$ as a product of reflections in $\left\{ r_1,\dots,r_{k-1} \right\}$.
\label{lem:reflection principle}
\end{lemma}
The reader may find it more convenient to think in terms of the symmetric group than in terms of the reflection group $A_{k-1}$:
if we denote for any permutation $\sigma\in\mathfrak S_k$ by $\sigma(\vec u)$ the lattice point obtained from $\vec u$ by permuting its coordinates according to $\sigma$, then \eqref{eq:reflection principle} may be written as
\[
P_n^+(\vec u\to\vec v) = \sum_{\sigma\in\mathfrak S_k}\sgn(\sigma)P_n(\sigma(\vec u)\to\vec v),
\]
where $\sgn(\sigma)$ denotes the sign of $\sigma$.
\begin{proof}
Since this is the fundamental result underlying the present manuscript, we repeat, for the reader's convenience, the complete proof.
The fundamental idea is to set up a (sign reversing) bijection between the set of $n$-step walks that, at some point, leave the Weyl chamber $\mathcal W$, thus showing that the total contribution of such ``bad walks'' to the right hand side of~\eqref{eq:reflection principle} is equal to zero.

For the following, it is important to impose a total ordering on the set of reflections given by \ref{eq:hyperplanes type A}.
This ordering can be chosen arbitrarily and we assume that $r_1\prec\dots\prec r_{k-1}$.
Now, the bijection is defined as follows.
Fix $r\in A_{k-1}$ and consider a walk $r(\vec u)\to\vec v$ with the step set $(\vec s^{(1)},\dots,\vec s^{(n)})$ that violates the condition of staying inside $\mathcal W$ (this is, in fact, always true for $r\neq\mathrm{id}$).
Clearly, this walk is contributing $(-1)^{l(r)}$ to the right hand side of \eqref{eq:reflection principle}.
By assumption on our walk, we know that there exists $0\le t<n$ such that $r(\vec u)+\sum_{j=1}^t\vec s^{(j)}\not\in\mathcal W$.
Let $t'$ denote the maximal $t$.
By condition 2 of definition~\ref{def:reflectable walks}, we know that $\vec c = r(\vec u)+\sum_{j=1}^{t'}\vec s^{(j)}$ is a point on the boundary of $\mathcal W$.
Consequently, this point belongs to some of the hyperplanes~\eqref{eq:hyperplanes type A}.
Let $r_i$ be the minimal $r_j$ (with respect to the order chosen above) such that $\vec c$ belongs to the corresponding hyperplane.
Having found $r_i$, we may now map the walk considered to the walk $(r_i\circ r)(\vec u))\to\vec v$ with the steps $(r_i(\vec s^{(1)}),\dots,r_i(\vec s^{(t')}),\vec s^{(t'+1)},\dots,\vec s^{(n)})$.
This simply defines a reflection at the hyperplane fixed by $r_i$ of the initial part of the walk up to the last contact $\vec c$ with the border of $\mathcal W$ and thus clearly defines a bijection.
The sign reversing property is easily checked by noting that $l(\cdot)$ has the property that $l(r_i\circ r)-l(r)\equiv 1\mod 2$.
%the constructed walk contributes $\sgn(r_i\circ\sigma)=-\sgn(\sigma)$ to the right hand side of \eqref{eq:reflection principle}.
Thus, the only surviving contributions to the right hand side of \eqref{eq:reflection principle} stem from walks $\vec u\to\vec v$ staying inside $\mathcal W$.
\end{proof}

The combination of lemma~\ref{lem:reflection principle} and generating functions give rise to an exact integral formula for $P_n^+(\vec u\to\vec v)$ that is amenable to asymptotic analysis.
For the sake of brevity, we introduce the following notion.
For $\vec a=(a_1,\dots,a_k)$ and $\vec z=(z_1,\dots,z_k)$ we set
\[ \vec z^\vec a=z_1^{a_1}z_2^{a_2}\cdots z_k^{a_k}. \]
(If the $a_j$'s were not integers, then we would have to resort to the interpretation in terms of formal exponentials.
Luckily, it will turn out below that in the cases relevant to this manuscript, this is not the case.)
Now, the \emph{step generating function} $S(\vec z)$ is defined as the Laurent polynomial
\[ S(\vec z) = \sum_{\vec s\in\mathcal S}\vec z^{\vec s}. \]
With this notion at hand, we see that
\begin{equation}
	P_n(\vec u\to\vec v)=\left[ \vec z^{\vec v} \right]\left( \vec z^{\vec u}S(\vec z)^n \right),
	\label{eq:n-step-walk genfun}
\end{equation}
where $\left[ \vec z^{\vec v} \right]$ means ``coefficient of $\vec z^{\vec v}$''.
Indeed, this is readily verified by noting that $S(\vec z)^n$ is equal to the sum of monomials $\vec z^{\vec v}$ such that $\vec v$ may be reached starting from $\vec 0$ with exactly $n$ steps from the set $\mathcal S$.
Almost trivially from \eqref{eq:n-step-walk genfun} and Cauchy's integral theorem follows the next lemma.
This expression -- up to a minor step -- can be found already in Gessel and Zeilberger~\cite[Theorem 2]{MR1092920}.
\begin{lemma}		
We have the exact expression
\begin{equation}
	P_n^+(\vec u\to\vec v)=\frac{1}{(2\pi i)^kk!}\idotsint\limits_{|z_1|=\dots=|z_k|=\rho}
	\det_{1\le j,m\le k}\left( z_j^{u_m} \right)\det_{1\le j,m\le k}\left( z_j^{-v_m} \right)S(z_1,\dots,z_k)^n\prod_{j=1}^k\frac{dz_j}{z_j}
	\label{eq:exact_integral_fixed_endpoint}
\end{equation}
\label{lem:exact_integral_fixed_endpoint}
\end{lemma}
\begin{proof}
	By virtue of Cauchy's integral theorem, the right hand side of \eqref{eq:n-step-walk genfun} may be written as the integral
	\[ \left( 2\pi i \right)^{-k}\idotsint S(\vec z)^n\frac{d\vec z}{\vec z^{\vec v-\vec u}}, \]
	where the integral is taken over the torus $|z_1|=\dots=|z_k|=\rho>0$.
	Substituting this expression for the corresponding term in \eqref{eq:reflection principle} and interchanging summation and integration we find
	\[
	P_n^+(\vec u\to\vec v)=\frac{1}{(2\pi i)^k}\idotsint \det_{1\le j,m\le k}\left( z_j^{u_m} \right)S(\vec z)^n\frac{d\vec z}{\vec z^{\vec v}}.
	\]
	This is \cite[Theorem 2]{MR1092920} for the type $A_{k-1}$ case.
	
	Now, for any permutation $\sigma$ we have the identity
	\[
	\det_{1\le j,m\le k}\left( z_{\sigma(j)}^{u_m} \right)S(\sigma(\vec z))^n\prod_{j=1}^kz_{\sigma(j)}^{-v_j}
	=
	\left(\sgn(\sigma^{-1})\prod_{j=1}^kz_j^{-v_{\sigma^{-1}(j)}} \right)\det_{1\le j,m\le k}\left( z_j^{u_m} \right)S(\vec z)^n.
	\]
	The reader should observe that the symmetry of $S(\vec z)$ is a consequence of condition $1$ in definition \ref{def:reflectable walks}.
	The proof of the lemma is now completed upon summing over all permutations $\sigma$ and dividing the result by $k!$.
\end{proof}

The last thing we need is some information on the structure of the step generating function $S(\vec z)$.
This, of course, boils down to the question: What step sets $\mathcal S$ do satisfy the conditions of definition~\ref{def:reflectable walks}?
The answer to this question has been given by Grabiner and Magyar~\cite{MR1235279}.
In their paper, they give, for any of the irreducible reflection groups, a complete classification of step sets satisfying the requirements of definition~\ref{def:reflectable walks}.
We state only the special case relevant to this manuscript.
\begin{lemma}[Grabiner and Magyar~\cite{MR1235279}]\label{lem:grabiner:classification}
	The step set $\mathcal S\subset\R^k\setminus\left\{ \vec 0 \right\}$ satisfies the conditions stated in lemma~\ref{lem:reflection principle} as well as assumption~\ref{assum:zero drift} if and only if $\mathcal S$ is (up to rescaling) equal either  to
\[
\mathcal S_\mathrm{rt} = \left\{ \pm\vec e^{(1)},\pm\vec e^{(2)},\dots,\pm\vec e^{(k)} \right\}
\qquad\textrm{or to}\qquad
\mathcal S_\mathrm{ls} = \left\{ \sum_{j=1}^{k}\varepsilon_j\vec e^{(j)}\ :\ \varepsilon_1,\dots,\varepsilon_k\in\left\{ -1,+1 \right\} \right\}, \]
where $\left\{ \vec e^{(1)},\dots,\vec e^{(k)} \right\}$ is the canonical basis of $\R^k$.
\end{lemma}
The subscripts of $\mathcal S_\mathrm{ls}$ and $\mathcal S_\mathrm{rt}$ refer to the \emph{lock step model} and the \emph{random turns model}, respectively, which point to the interpretation of these lattice walk models in terms of non-intersecting lattice paths (see figures~\ref{fig:lock step model} and \ref{fig:random turns model} for illustrations).
\begin{figure}
\begin{tikzpicture}[>=stealth]
	\scope
	\clip (-.4,-.4) rectangle (5.4,5.4);
	\fill[blue!5!white] (-.4,-.4) -- (10,10) -- (-.4,10) -- cycle; 
	\draw[step=1cm,gray,very thin,rotate=45,scale=0.70710678119] (-10,-10) grid (12,12);
	\endscope
	\draw[->,thin] (-.4,2) -- (5.7,2) node[anchor=west]  {$x_1$};
	\draw[->,thin] (2,-.4) -- (2,5.7) node[anchor=south]  {$x_2$};
	\draw[blue,thick] (-.5,-.5) -- (5.4,5.4) {};
	\scope[->,ultra thick,color=red]
	\draw (1,2) -- (1.5,2.5);
	\draw (1.5,2.5) -- (1,3);
	\draw (1,3) -- (1.5,3.5);
	\draw (1.5,3.5) -- (2,3);
	\draw (2,3) -- (2.5,3.5);
	\draw (2.5,3.5) -- (2,4);
	\draw (2,4) -- (1.5,4.5);
	\draw (1.5,4.5) -- (1,4);
	\draw (1,4) -- (.5,3.5);
	\draw (.5,3.5) -- (1,3);
	\endscope
	\fill[fill=blue] (1,2) node[anchor=north east,blue,thick] {$S$} circle (.1cm);
	\fill[fill=blue] (1,3) node[anchor=north east,blue,thick] {$E$} circle (.1cm);

	\scope[xshift=7cm]
	\scope
	\clip (-.4,-.4) rectangle (5.4,5.4);
	\draw[step=1cm,gray,very thin,rotate=45,scale=0.70710678119] (-10,-10) grid (12,12);
	\endscope
	%\draw[->,thin] (6.6,0) -- (12.7,0) node[anchor=west]  {$n$};
	\draw[->,thin] (-.4,2) -- (5.7,2) node[anchor=west]  {$n$};
	%\draw[->,thin] (7,-.4) -- (7,5.7) node[anchor=south]  {$x_1, x_2$};
	\draw[->,thin] (0,-.4) -- (0,5.7) node[anchor=south]  {$x_1, x_2$};
	\scope[ultra thick,color=black]
	\draw (0,1) -- (.5,1.5) -- (1,1) -- (1.5,1.5) -- (2,2) -- (2.5,2.5) -- (3,2) -- (3.5,1.5) -- (4,1) -- (4.5,.5) -- (5,1); 
	\draw (0,2) -- (.5,2.5) -- (1,3) -- (1.5,3.5) -- (2,3) -- (2.5,3.5) -- (3,4) -- (3.5,4.5) -- (4,4) -- (4.5,3.5) -- (5,3);
	\endscope
	\endscope
\end{tikzpicture}
\caption{
Example of a lattice walk in the lock step model and its corresponding interpretation in terms of non-intersecting lattice paths.
On the left: a $10$-step walk from $S=(-2,0)$ to $E=(-2,2)$ restricted to the Weyl chamber $0<x_1<x_2$ (indicated by the shaded region).
On the right: the corresponding pair of non-intersecting lattice paths: the lower path from $(0,-2)$ to $(10,-2)$ keeps track of the horizontal coordinate of the walk in the left hand side, while the upper path from $(0,0)$ to $(10,2)$ keeps track of the vertical coordinate.
}
\label{fig:lock step model}
\end{figure}
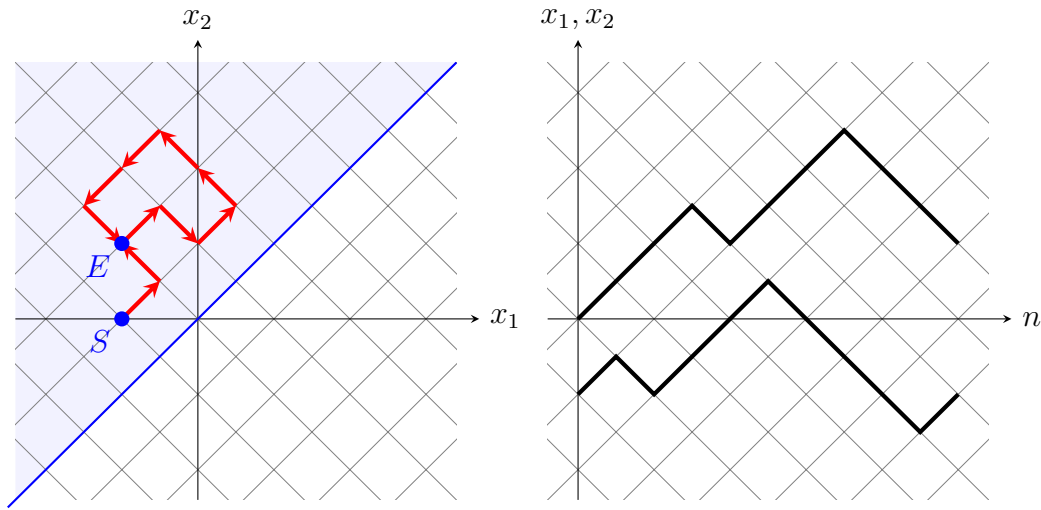
\begin{figure}
\begin{tikzpicture}[>=stealth]
	\scope
	\clip (-.4,-.4) rectangle (5.4,5.4);
	\fill[blue!5!white] (-.4,-.4) -- (10,10) -- (-.4,10) -- cycle; 
	\draw[step=1cm,gray,very thin] (-10,-10) grid (12,12);
	\endscope
	\draw[->,thin] (-.4,2) -- (5.7,2) node[anchor=west]  {$x_1$};
	\draw[->,thin] (2,-.4) -- (2,5.7) node[anchor=south]  {$x_2$};
	\draw[blue,thick] (-.5,-.5) -- (5.4,5.4) {};
	\scope[->,ultra thick,color=red]
	\draw (0,1) -- (0,2);
	\draw (0,2) -- (1,2);
	\draw (1,2) -- (1,3);
	\draw (1,3) -- (1,4);
	\draw (1,4) -- (1,5);
	\draw (1,5) -- (2,5);
	\draw (2,5) -- (2,4);
	\draw (2,4) -- (1,4);
	\draw (1,4) -- (0,4);
	\draw (0,4) -- (0,3);
	\endscope
	\fill[fill=blue] (0,1) node[anchor=north east,blue,thick] {$S$} circle (.1cm);
	\fill[fill=blue] (0,3) node[anchor=north east,blue,thick] {$E$} circle (.1cm);

	\scope[xshift=7cm]
	\scope
	\clip (-.4,-.4) rectangle (5.4,5.4);
	\draw[step=0.5cm,gray,very thin,rotate=45,scale=0.70710678119] (-10,-10) grid (12,12);
	\draw[step=0.5cm,gray,very thin] (-10,-10) grid (12,12);
	\endscope
	%\draw[->,thin] (6.6,0) -- (12.7,0) node[anchor=west]  {$n$};
	\draw[->,thin] (-.4,2) -- (5.7,2) node[anchor=west]  {$n$};
	%\draw[->,thin] (7,-.4) -- (7,5.7) node[anchor=south]  {$x_1, x_2$};
	\draw[->,thin] (0,-.4) -- (0,5.7) node[anchor=south]  {$x_1, x_2$};
	\scope[ultra thick,color=black]
	\draw (0,1) -- (.5,1) -- (1,1.5) -- (1.5,1.5) -- (2,1.5) -- (2.5,1.5) -- (3,2) -- (3.5,2) -- (4,1.5) -- (4.5,1) -- (5,1);
	\draw (0,1.5) -- (.5,2) -- (1,2) -- (1.5,2.5) -- (2,3) -- (2.5,3.5) -- (3,3.5) -- (3.5,3) -- (4,3) -- (4.5,3) -- (5,2.5);
	\endscope
	\endscope
\end{tikzpicture}
\caption{
Example of a lattice walk in the random turns model and its corresponding interpretation in terms of non-intersecting lattice paths.
On the left: a $10$-step walk from $S=(-2,-1)$ to $E=(-2,1)$ restricted to the Weyl chamber $0<x_1<x_2$ (indicated by the shaded region).
On the right: the corresponding pair of non-intersecting lattice paths: the lower path from $(0,-2)$ to $(10,-2)$ keeps track of the horizontal coordinate of the walk in the left hand side, while the upper path from $(0,-1)$ to $(10,1)$ keeps track of the vertical coordinate.
}
\label{fig:random turns model}
\end{figure}
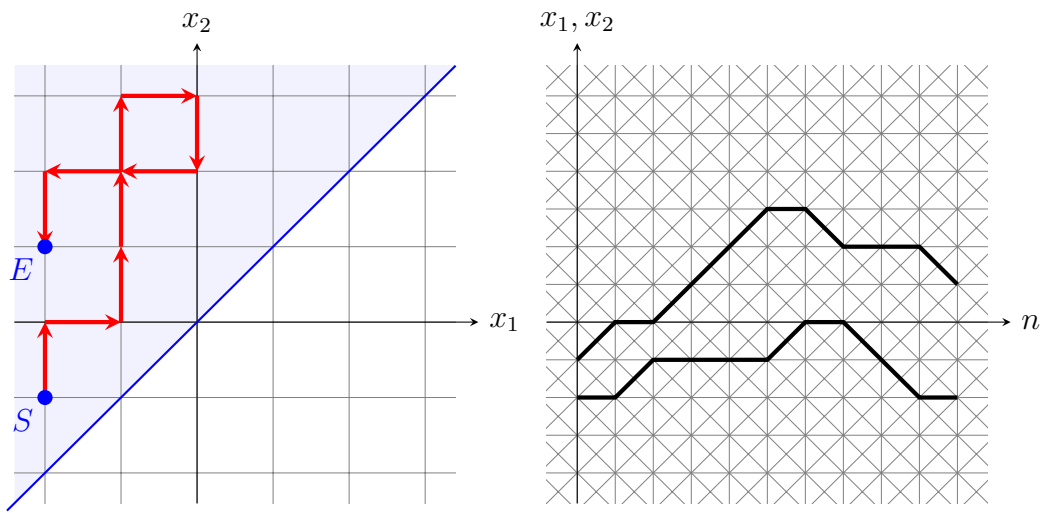

The corresponding step generation functions are given by
\[S_\mathrm{ls}(\vec z)=\prod_{j=1}^k\left( z_j+\frac{1}{z_j} \right)\qquad\textrm{or}\qquad  S_\mathrm{rt}(\vec z)=\sum_{j=1}^k\left( z_j+\frac{1}{z_j}\right). \]
We close this section with the following simple asymptotic result.
\begin{lemma}
	We have the asymptotics
	\[
	\log\left| S\left( e^{i\varphi_1},\dots,e^{i\varphi_k} \right)\right|=
	\log S(\vec 1)-\frac{\Lambda}{2}\sum_{j=1}^k\varphi_j^2+\frac{\Omega}{8}\left( \sum_{j=1}^k\varphi_j^2 \right)^2+\frac{\Psi}{4!}\sum_{j=1}^k\varphi_j^4+O\left( \max_j|\varphi_j|^6 \right)
	\]
	as $(\varphi_1,\dots,\varphi_k)\to(0,\dots,0)$, where either
	\begin{equation*}
		S = \Sls, \qquad \Lambda = 1,\ \Omega = 0 \ \textrm{and}\ \Psi=-2
	\end{equation*}
	or 
	\begin{align*}
		S &= \Srt, \qquad \Lambda = \frac{1}{k},\ \Omega =-\frac{1}{k^2} \ \textrm{and}\ \Psi=\frac{1}{k}.
	\end{align*}
	\label{lem:step generating function asymptotics}
\end{lemma}

\subsection{Asymptotics for walks with a fixed end point}
We derive asymptotics for the quantity $P_n^+(\vec u\to\vec v)$ as $n\to\infty$ in the set $\left\{ n\ :\ P_n^+(\vec u\to\vec v)>0 \right\}$.
It may be obvious to the reader already from \eqref{eq:exact_integral_fixed_endpoint} how to determine asymptotics for $P_n^+(\vec u\to\vec v)$, as the integral on the right hand side is a multi-dimensional version of integrals of the form
\[ \frac{1}{2\pi i}\oint_{|z|=1} g(z)f(z)^n\frac{d z}{z}. \]
Under suitable assumptions on $g$ and $f$, this integral can be asymptotically analysed by means of the saddle point method.
The reason for this being that, as $n\to\infty$, the integral is largely dominated by the maxima of $|f(z)|$ on the contour of integration.
Let us for the moment assume that $z=1$ is the only maximum of $|f(z)|$ on the integration contour, and moreover, we assume that $z=1$ is a saddle point of $f(z)$, i.e., $f'(1)=0$ (this is important and may, in general, require a change of the contour of integration).
Then, under suitable conditions, we will have
\[
\frac{1}{2\pi i}\oint_{|z|=1} g(z)f(z)^n\frac{d z}{z} \sim
\frac{1}{2\pi}\int\limits_{-\varepsilon}^{\varepsilon}g\left( e^{i\varphi} \right)f\left( e^{i\varphi} \right)^nd\varphi,
\qquad n\to\infty,
\]
where $\varepsilon=\varepsilon(n)\to 0$ as $n\to\infty$.
If $\varepsilon\to 0$ fast enough such that $n\varepsilon^2\to \infty$ but $n\varepsilon^3\to 0$, then we shall have, assuming that $g(1)\neq 0$, 
\begin{align*}
\frac{1}{2\pi i}\oint_{|z|=1} g(z)f(z)^n\frac{d z}{z}
&\sim \frac{g(1)f(1)^n}{2\pi}\int\limits_{-\varepsilon}^{\varepsilon}e^{-nf''(1)\varphi^2/(2f(1))}d\varphi \\
&\sim \frac{g(1)f(1)^n}{2\pi\sqrt{n f''(1)/f(1)}}\int\limits_{-\infty}^\infty e^{-\varphi^2/2}d\varphi
= \frac{g(1)f(1)^n}{\sqrt{2\pi n f''(1)/f(1)}}
\end{align*}
as $n\to\infty$.
A more in-depth discussion of the saddle point method can for example be found in \cite{MR671583}.

The analysis of \eqref{eq:exact_integral_fixed_endpoint} is a bit more delicate, mainly for two reasons.
First, it turns out that the asymptotics is governed by more than one saddle point.
Second, the product of the two determinants, which corresponds to the function $g$ above, evaluated at these saddle points will always be equal to zero.
We therefore cannot just replace it by its value at the saddle point but we have to determine (sufficiently accurate) Taylor approximations thereof.
But this certainly poses a non-trivial problem because of a typically large number of cancellations induced by the nature of the determinants.
We will address this problem, the solution of which is fundamental to the asymptotic analysis of \eqref{eq:exact_integral_fixed_endpoint}, in the next section.

\begin{theorem}
	In the lock step model, we have the asymptotics 
	\begin{multline}
		P_n^+(\vec u\to\vec v)=
		\frac{2^{kn}}{n^{k^2/2}}\left( \frac{2}{\pi} \right)^{k/2}
		\frac{\prod\limits_{1\le j<m\le k}(u_m-u_j)(v_m-v_j)}{\prod\limits_{j=1}^k(j-1)!}
		\\ \times \left(
			1
			+\frac{1}{kn}\left( \sum_{j=1}^ku_j \right)\left( \sum_{j=1}^kv_j \right)
			-\frac{1}{2n}\left( \sum_{j=1}^k(u_j^2+v_j^2) \right)
			+O\left( n^{-5/3} \right)
		\right)
		\label{eq:lockstep:asymptotics:u->v}
	\end{multline}
	\label{thm:lockstep:asymptotics:u->v}
	as $n\to\infty$ in the set $\left\{ n\ :\ P_n^+(\vec u\to\vec v)>0 \right\}$.
	\end{theorem}

\begin{theorem}
	In the random turns model, we have the asymptotics 
	\begin{multline}
		P_n^+(\vec u\to\vec v)=
		\left( \frac{k^k}{2\pi} \right)^{k/2}
		\frac{(2k)^{n}}{n^{k^2/2}}
		\frac{\prod\limits_{1\le j<m\le k}(u_m-u_j)(v_m-v_j)}{\prod\limits_{j=1}^k(j-1)!}
		\\ \times \left(
			1
			+\frac{1}{k^2n}\left( \sum_{j=1}^ku_j \right)\left( \sum_{j=1}^kv_j \right)
			-\frac{1}{2kn}\left( \sum_{j=1}^k(u_j^2+v_j^2) \right)
			+O\left( n^{-5/3} \right)
		\right)
		\label{eq:randomturns:asymptotics:u->v}
	\end{multline}
	\label{thm:randomturns:asymptotics:u->v}
	as $n\to\infty$ in the set $\left\{ n\ :\ P_n^+(\vec u\to\vec v)>0 \right\}$.
\end{theorem}

\subsection{Asymptotics for walks with a free end point}

The number $P_n^+(\vec u)$ of $n$ step walks confined to $\mathcal W$ starting in $\vec u\in\mathcal L\cap \mathcal W$ with a free end point is given by the finite sum
\[ P_n^+(\vec u)=\sum_{\vec v\in\mathcal L\cap\mathcal W}P_n^+(\vec u\to\vec v). \]
An approach analogously to that used in the proofs of theorem~\ref{thm:lockstep:asymptotics:u->v} and \ref{thm:randomturns:asymptotics:u->v} can also be used to establish asymptotics ast $n\to\infty$. (Although, summation introduces some additional obstacles.)

\begin{theorem}
	In the lock step model, the number of walks in $x_1<x_2<\dots<x_k$ of length $n$ starting in $\vec u$ is given by
	\begin{multline}
		P_n^+(\vec u)=
		\frac{2^{k n+2\ell}}{\pi^{\ell/2} n^{\frac{1}{2}\binom{k}{2}}}\left( \prod_{j=0}^{\ell}j!\frac{\Gamma(j+\alpha)}{\sqrt\pi} \right) 
		\frac{\prod\limits_{1\le j<m\le k}(u_m-u_j)}{\prod\limits_{j=1}^kj!}\\
		\times\left( 
		1
		+\frac{\ell(\ell+\alpha-1)}{n}\left( \frac{(k-2)!}{(k+1)!}\sum_{1\le j<m\le k}(u_m-u_j)^2-\frac{2\ell+\alpha-3}{6} \right)
		+O\left( n^{-3/2} \right)
		\right)
		\label{eq:lockstep:asymptotics:u->}
	\end{multline}
	as $n\to\infty$, where $\ell=\floor{k/2}$ and $\alpha=\frac{1}{2}+k-2\ell$.
	\label{thm:lockstep:asymptotics:u->}
\end{theorem}

\begin{theorem}
	In the random turns model, the number of walks in $x_1<x_2<\dots<x_k$ of length $n$ starting in $\vec u$ is given by
	\begin{multline}
		P_n^+(\vec u)=
		\frac{(2k)^n 4^\ell}{\pi^{\ell/2} (n/k)^{\frac{1}{2}\binom{k}{2}}}\left( \prod_{j=0}^{\ell}j!\frac{\Gamma(j+\alpha)}{\sqrt\pi} \right) 
		\frac{\prod\limits_{1\le j<m\le k}(u_m-u_j)}{\prod\limits_{j=1}^kj!}\\
		\times\Bigg( 
		1
		+\frac{\ell(\ell+\alpha-1)}{n k}\Bigg( \frac{(k-2)!}{(k+1)!}\sum_{1\le j<m\le k}(u_m-u_j)^2+\frac{2k-3}{24}\\
		-\frac{1+\ell(k-\ell-\frac{1}{2})}{2k} \Bigg)
		+O\left( n^{-3/2} \right)
		\Bigg)
		\label{eq:randomturns:asymptotics:u->}
	\end{multline}
	as $n\to\infty$, where $\ell=\floor{k/2}$ and $\alpha=\frac{1}{2}+k-2\ell$.
	\label{thm:randomturns:asymptotics:u->}
\end{theorem}
\section{Asymptotics for determinants}
\label{sec:asymptotics and determinants}
Let $f(w)$ be a function analytic around $w=0$ having Taylor series expansion
\[ f(w) = \sum_{j=0}^{\infty}a_n w^n, \qquad |w|<R, \]
with positive radius of convergence $R>0$.

In this manuscript we will frequently be concerned with functions of the form
\[ F(w_1,\dots,w_k)=\det_{1\le j,m\le k}\left( f(w_ju_m) \right), \]
where $u_1,\dots,u_k$ are some fixed constants.
Clearly, $F(\vec w)=F(w_1,\dots,w_k)$ is analytic in $|\vec w|_{\infty}=\max_j|w_j|<R/|\vec u|_\infty$.
But can we say something about $F(w_1,\dots,w_k)$ as $(w_1,\dots,w_k)\to(0,\dots,0)$?
Obviously, if $k>1$ then $F(w_1,\dots,w_k)\to 0$.
But this, of course, is very imprecise.
Fortunately, it is possible to write down the complete Taylor series expansion of $F(w_1,\dots,w_k)$
(see Lemma~\ref{lem:det asymptotics} below).

Before actually stating the lemma, we need to introduce so called Schur functions.

For a partition $\vec\mu=(\mu_1,\dots,\mu_k)$, we define the Schur function $s_{\vec\mu}(w_1,\dots,w_k)$ by
\[ s_{\vec\mu}(\vec w) = \frac{\det\limits_{1\le j,m\le k}\left( w_j^{\mu_m+k-m} \right)}{\det\limits_{1\le j,m\le k}\left( w_m^{k-m} \right)}. \]
It can be readily checked that $s_{\vec\mu}(\vec w)$ is a symmetric homogenous polynomial (the denominator is - upon rearranging - a Vandermonde determinant and each zero of the denominator is also a zero of the numerator).

\begin{lemma}
	Let $f(w)=\sum_{j=0}^\infty a_jw^j$ be an analytic function for $|w|<R$.
	Then, for any parameters $\vec u = (u_1,\dots,u_k)$, the function
	\[ F(\vec z)=F(z_1,\dots,z_k)=\det_{1\le j,m\le k}\left( f(z_ju_m) \right) \]
	is analytic for $|\vec z|_\infty<R/|\vec u|_\infty$.
	Furthermore, we have the series expansion
\begin{multline*}
	F(\vec z) = \det_{1\le j,m\le k}\left( z_j^{m-1} \right)\det_{1\le j,m\le k}\left( u_j^{m-1} \right) \\
	\times \sum_{\mu_1\ge \dots\ge\mu_k\ge 0}\left( \prod_{j=1}^ka_{\mu_j+k-j} \right)s_{(\mu_1,\dots,\mu_k)}(u_1,\dots,u_k)s_{(\mu_1,\dots,\mu_k)}(z_1,\dots,z_k),
\end{multline*}
convergent for $|\vec z|_\infty<R/|\vec u|_\infty$.
	\label{lem:det asymptotics}
\end{lemma}
\begin{proof}
The proof of this lemma is actually a quite simple one and goes as follows.
Let $\lambda_1,\dots,\lambda_k$ be non-negative integers.
Then the coefficient of $\prod z_j^{\lambda_j}$ is given by
\[
\left[ \prod_{j=1}^kz_j^{\lambda_j} \right]\det_{1\le j,m\le k}\left( f(z_ju_m) \right)=\left( \prod_{j=1}^ka_{\lambda_j} \right)\det_{1\le j,m\le k}\left( u_m^{\lambda_j} \right).
\]
(This is easily seen to hold true as $z_j$ occours only in the $j$-th row.)
Clearly, this expression is eqal to zero whenever $\lambda_1,\dots,\lambda_k$ are not pairwise different (as in this case the determinant on the right hand side is equal to zero).
We may therefore assume that each $\lambda_j$ is different from all the others.
Now, the crucial observation is the following: if $\sigma$ is any permutation on $\left\{ 1,2,\dots,k \right\}$ (the set of all such permutations is denoted by $\mathfrak{S}_k$), then the coefficients of $\prod z_j^{\lambda_j}$ and $\prod z_j^{\lambda_{\sigma(j)}}$ differ only by sign.
This sign is equal to $\sgn(\sigma)$, as can be readily verified upon rearrangement of the determinant on the right hand side above.
Hence, we see that
\begin{multline*}
	\sum_{\sigma\in\mathfrak{S}_k}\left( \left[ \prod_{j=1}^kz_j^{\lambda_{\sigma(j)}} \right]\det_{1\le j,m\le k}\left( f(z_ju_m) \right) \right)\left( \prod_{j=1}^kz_j^{\lambda_{\sigma(j)}} \right)
	= \left( \prod_{j=1}^k a_{\lambda_j} \right)\det_{1\le j,m\le k}\left( u_m^{\lambda_j} \right)\det_{1\le j,m\le k}\left( z_j^{\lambda_m} \right).
\end{multline*}
Obviously, this expression is invariant under permutations of the $\lambda_1,\dots,\lambda_k$.
We may therefore assume that $\lambda_1>\dots>\lambda_k$.
Now, setting $\lambda_j=\mu_j+k-j$, noting that
\[ 
\det_{1\le j,m\le k}\left( u_m^{\lambda_j} \right)\det_{1\le j,m\le k}\left( z_j^{\lambda_m} \right)
=
\det_{1\le j,m\le k}\left( z_j^{m-1} \right)\det_{1\le j,m\le k}\left( u_j^{m-1} \right)
s_{(\mu_1,\dots,\mu_k)}(\vec z)s_{(\mu_1,\dots,\mu_k)}(\vec u)
\]
and summing over all $\mu_1\ge \dots\ge\mu_k$ we obtain the claimed result.
\end{proof}

\begin{example}
	Set $f(w)=e^{i w}$.
	Then, by lemma~\ref{lem:det asymptotics} above, we have the asymptotics
	\begin{multline*}
		\det_{1\le j,m\le k}\left( e^{i u_mz_j} \right)
		=
		i^{\binom{k}{2}}
		\left(\prod\limits_{j=1}^k(j-1)!\right)^{-1}
		\det\limits_{1\le j,m\le k}\left( u_j^{m-1} \right)\det_{1\le j,m\le k}\left( z_j^{m-1} \right) \\
		\times
		\left(
			1
			+\frac{i}{k}\left( \sum_{j=1}^ku_j \right)\left( \sum_{j=1}^kz_j \right)
			-\frac{1}{(k-1)k}\left( \sum_{j<m}u_ju_m \right)\left( \sum_{j<m}z_jz_m \right) \right.\\
			\left.
			-\frac{1}{k(k+1)}\left( \sum_{j\le m}u_ju_m \right)\left( \sum_{j\le m}z_jz_m \right)
			+O\left( \max_j|z_j|^{3} \right)
		\right)
	\end{multline*}
	as $(z_1,\dots,z_k)\to(0,\dots,0)$.
	\label{ex:det expansion}
\end{example}

\begin{example}[Cauchy's identity]
	Setting $f(w)=1/(1-w)$ in lemma~\ref{lem:det asymptotics} yields
	\[
	\det_{1\le j,m\le k}\left( \frac{1}{1-x_jy_m} \right)
	=
	\det_{1\le j,m\le k}\left( x_j^{m-1} \right)\det_{1\le j,m\le k}\left( y_j^{m-1} \right)
  \sum_{\mu_1\ge \dots\ge\mu_k\ge 0}s_{(\mu_1,\dots,\mu_k)}(\vec x)s_{(\mu_1,\dots,\mu_k)}(\vec y).
	\]
	Now, comparing this equation with (a variant of) Cauchy's double alternante (see \cite[p. 311]{47.0878.03}), viz.
	\[
	\det_{1\le j,m\le k}\left( \frac{1}{1-x_jy_m} \right)
	=
	\frac{\det\limits_{1\le j,m\le k}\left( x_j^{m-1} \right)\det\limits_{1\le j,m\le k}\left( y_j^{m-1} \right)}{\prod\limits_{j,m=1}^k(1-x_jy_m)},
	\]
	we recover Cauchy's classical identity
	\[
	\frac{1}{\prod\limits_{j,m=1}^k(1-x_jy_m)} 
	=
  \sum_{\mu_1\ge \dots\ge\mu_k\ge 0}s_{(\mu_1,\dots,\mu_k)}(\vec x)s_{(\mu_1,\dots,\mu_k)}(\vec y).
	\]
\end{example}
\section{Fixed end points: Proofs}
\label{sec:proofs:fixed}
We start from the integral representation
\[
P_n^+(\vec u\to\vec v)=\frac{1}{(2\pi)^k k!}\int_{-\pi}^{\pi}\dots\int_{-\pi}^\pi \det_{1\le j,m\le m}\left( e^{i\varphi_ju_m} \right)\det_{1\le j,m\le m}\left( e^{-i\varphi_jv_m} \right)S\left( e^{i\varphi_1},\dots,e^{i\varphi_k} \right)^n\prod_{j=1}^kd\varphi_j.
\]
This integral almost suggests itself to the saddle point approach, as it is roughly of the form
\[ \int g(x)f(x)^nd x \]
for some well-behaved functions $g$ and $f$.
(Admittedly, our case is a bit more involved as the function $g$ is defined in terms of certain determinants.
But we have already shown in a previous section how to deal with such functions, so we do not expect too many troubles ahead.)

The dominant asymptotic behaviour of $P_n^+(\vec u\to\vec v)$ is therefore expected to be captured by small neighbourhoods of the maximal points of $|S(e^{i\varphi_1},\dots,e^{i\varphi_k})|$ on $(-\pi,\pi]^k$.
The set of maximal points will be denoted by $\mathcal M$, i.e.,
\[
\mathcal M=\left\{ (\varphi_1,\dots,\varphi_k)\in(-\pi,\pi]^k\ :\ |S(e^{i\varphi_1},\dots,e^{i\varphi_k})|=|S(1,\dots,1)| \right\}.
\]

For the sake of convenience, we define for $\varepsilon>0$ the sets
\[
\mathcal U_\varepsilon(\hat{\vec\varphi})=\left\{ \vec\varphi\in\R^k\ :\ |\hat{\vec\varphi}-\vec\varphi|_\infty<\varepsilon \right\},
\qquad \hat{\vec\varphi}=(\hat\varphi_1,\dots,\hat\varphi_k)\in\mathcal M.
\]
We claim (and prove in the following) that the asymptotic dominant behaviour of $P_n^+(\vec u\to\vec v)$ as $n\to\infty$ is captured by 
\begin{equation}
\frac{1}{(2\pi)^kk!}\sum_{\hat{\vec\varphi}\in\mathcal M}\idotsint\limits_{\mathcal U_\varepsilon(\hat{\vec\varphi})}
\det_{1\le j,m\le m}\left( e^{i\varphi_ju_m} \right)\det_{1\le j,m\le m}\left( e^{-i\varphi_jv_m} \right)S\left( e^{i\varphi_1},\dots,e^{i\varphi_k} \right)^n\prod_{j=1}^kd\varphi_j,
\label{eq:dominant part:fixed end point}
\end{equation}
where we may choose $\varepsilon=n^{-5/12}$.

Now, before actually asymptotically evaluating \eqref{eq:dominant part:fixed end point}, let us find a bound for the complementary part of the integral.
For convenience, set $\mathcal U_\varepsilon=\bigcup_{\hat{\vec\varphi}\in\mathcal M}\mathcal U_\varepsilon(\hat{\vec\varphi})$.
Clearly, $\det\left( e^{iu_m\varphi_j} \right)$ as well as $\det\left( e^{-v_m\varphi_j} \right)$ are bounded on $(-\pi,\pi]^k$ as $n\to\infty$.
On the other hand, $|S(e^{i\varphi_1},\dots,e^{i\varphi_k})|^n$ restricted to $(-\pi,\pi]^k\setminus\mathcal U_\varepsilon$ attains, at least for $n$ large enough, its maximum somewhere on the boundary of one of the sets $\mathcal U_\varepsilon(\hat{\vec\varphi})$ for some point $\hat{\vec\varphi}\in\mathcal M$.
Let $\vec\psi=(\psi_1,\dots,\psi_k)\in(-\pi,\pi]^k\setminus\mathcal U_\varepsilon$ be such a point maximising $|S(e^{i\varphi_1},\dots,e^{i\varphi_k})|$ on this set.
By the above considerations, we know that there exists $\tilde{\vec\varphi}\in\mathcal M$ such that $|\vec\psi-\tilde{\vec\varphi}|=\varepsilon$ (at least for $n$ sufficiently large).
Hence, lemma~\ref{lem:step generating function asymptotics} shows that there exists a constant $C^*>0$ such that
\[
\left|S\left( e^{i\varphi_1},\dots,e^{i\varphi_k} \right)\right|^n
\le \left|S\left(e^{i\psi_1},\dots,e^{i\psi_k}\right)\right|^n
=S(1,\dots,1)^{n-C^*n^{1/6}+O\left( n^{-2/3} \right)}
\]
valid for $(\varphi_1,\dots,\varphi_k)\in(-\pi,\pi]^k\setminus\mathcal U_\varepsilon$ as $n\to\infty$.
Trivial bounds for the integral then show that
\begin{multline*}
\frac{1}{\pi^kk!}\idotsint\limits_{(-\pi,\pi]^k\setminus\mathcal U_\varepsilon}
\det_{1\le j,m\le m}\left( e^{i\varphi_ju_m} \right)\det_{1\le j,m\le m}\left( e^{-i\varphi_jv_m} \right)S\left( e^{i\varphi_1},\dots,e^{i\varphi_k} \right)^n\prod_{j=1}^kd\varphi_j  \\
=O\left( S(1,\dots,1)^{n-Cn^{1/6}} \right)
\end{multline*}
as $n\to\infty$.
This is exponentially small compared to \eqref{eq:dominant part:fixed end point}, which is, as we will see below, of order $n^{-k^2/2}S(1,\dots,1)^{n}$ as $n\to\infty$.

Let us now turn our attention to \eqref{eq:dominant part:fixed end point}, and determine asymptotics as $n\to\infty$.
As the first important observation we note that for any $\hat{\vec\varphi}=(\hat{\varphi}_1,\dots,\hat{\varphi}_k)\in\mathcal M$ we have
\begin{multline}
\det_{1\le j,m\le m}\left( e^{i(\hat{\varphi}_j+\varphi_j)u_m} \right)\det_{1\le j,m\le m}\left( e^{-i(\hat{\varphi}_j+\varphi_j)v_m} \right)S\left( e^{i(\hat{\varphi}_1+\varphi_1)},\dots,e^{i(\hat{\varphi}_k+\varphi_k)} \right)^n
\\=
%(-1)^{\left( n+u_1+v_1 \right)\sum_{j=1}^{k}\varphi_j}
\det_{1\le j,m\le m}\left( e^{i\varphi_ju_m} \right)\det_{1\le j,m\le m}\left( e^{-i\varphi_jv_m} \right)S\left( e^{i\varphi_1},\dots,e^{i\varphi_k} \right)^n.
\label{eq:fixed endpoint:det factorisation}
\end{multline}
This identity is the only part of the proof where we really have to distinguish wether we are in the lock step model or in the random turns model.
If we have $S=\Sls$, i.e., the lock step model, then we know that $P_n^+(\vec u\to\vec v)>0$ if and only if $u_1\equiv\dots\equiv u_k \mod 2$ and $v_1\equiv\dots\equiv\mod 2$ as well as $u_1-v_1\equiv n\mod 2$.
We therefore may pull the factors $e^{i\hat\varphi_ju_1}$ and $e^{-i\hat\varphi_jv_1}$ out of the respective determinants on the left hand side in a row wise fashion.
On the other hand, if $S=\Srt$ (random turns model), then we know that $\hat\varphi_1=\dots=\hat\varphi_k\in\left\{ 0,\pi \right\}$, and we may pull the factors $e^{i\hat\varphi_1 u_m}$ and $e^{-i\hat\varphi_1 v_m}$ out of the respective determinants in a column wise fashion.
The validity of the equation above is then seen by recalling that $n=\sum_{j=1}^k(v_j-u_j)$ whenever $P_n^+(\vec u\to\vec v)>0$.

But this shows that all the saddle points $\mathcal M$ contribute precisely the same value to \eqref{eq:dominant part:fixed end point}.
Hence, \eqref{eq:dominant part:fixed end point} is equal to
\begin{equation}
\frac{|\mathcal M|}{(2\pi)^kk!}\int\limits_{-\varepsilon}^\varepsilon\!\!\dots\!\!\int\limits_{-\varepsilon}^\varepsilon
\det_{1\le j,m\le m}\left( e^{i\varphi_ju_m} \right)\det_{1\le j,m\le m}\left( e^{-i\varphi_jv_m} \right)S\left( e^{i\varphi_1},\dots,e^{i\varphi_k} \right)^n\prod_{j=1}^kd\varphi_j,
\label{eq:fixed end point:h1}
\end{equation}
where $|\mathcal M|$ denotes the cardinality of the set $\mathcal M$.
Asymptotics for this integral as $n\to\infty$ can be established by replacing the integrand with (sufficiently accurate) Taylor series approximations.
This is the second step of the saddle point method.
For the sake of convenience, we define
\[
\selberg{f(\vec\varphi)}_\varepsilon=\int\limits_{-\varepsilon}^\varepsilon\!\!\dots\!\!\int\limits_{-\varepsilon}^\varepsilon f(\vec\varphi)\left( \prod_{1\le j<m\le k}(\varphi_m-\varphi_j) \right)^2e^{-n\Lambda\sum\limits_{j=1}^k\varphi_j^2/2}
\prod_{j=1}^kd\varphi_j.
\]
With this notation at hand, \eqref{eq:fixed end point:h1} is seen to be asymptotically equal to
\begin{multline*}
	\frac{|\mathcal M|S(1,\dots,1)^n}{(2\pi)^kk!}
	\frac{\left( \prod\limits_{1\le j<m\le k}(u_m-u_j) \right)\left( \prod\limits_{1\le j<m\le k}(v_m-v_j) \right)}{\left( \prod\limits_{j=1}^k(j-1)! \right)^2}\\
	\times\Biggl( 
	\selberg{1}_\varepsilon
	+\frac{s_{(1)}(\vec u)s_{(1)}(\vec v)}{k^2}\selberg{s_{(1)}(\vec\varphi)^2}_\varepsilon
%	+\frac{1}{k^2}\left( \sum_{j=1}^ku_j \right)\left( \sum_{j=1}^kv_j \right)\selberg{\left(\sum_{j=1}^k\varphi_j \right)^2}_\varepsilon
-\frac{s_{(1,1)}(\vec u)+s_{(1,1)}(\vec v)}{(k-1)k}\selberg{s_{(1,1)}(\vec\varphi)}_\varepsilon
%	-\frac{1}{(k-1)k}\left( \sum_{1\le j<m\le k}(u_ju_m+v_jv_m) \right)\left( \sum_{1\le j<m\le k}\varphi_j\varphi_m \right)
-\frac{s_{(2)}(\vec u)+s_{(2)}(\vec v)}{k(k+1)}\selberg{s_{(2)}(\vec\varphi)}_\varepsilon
%	-\frac{1}{k(k+1)}\left( \sum_{1\le j\le m\le k}(u_ju_m+v_jv_m) \right)\left( \sum_{1\le j\le m\le k}\varphi_j\varphi_m \right)
\\
+n\frac{\Omega}{8}\selberg{\left(\sum_{j=1}^k\varphi_j^2\right)^2}_\varepsilon
+n\frac{\Psi}{4!}\selberg{\sum_{j=1}^k\varphi_j^4}_\varepsilon
+\selberg{O\left(n \max_j|\varphi_j|^6\right)}_\varepsilon
	\Biggl)
\end{multline*}
as $n\to\infty$, where the constants $\Lambda$, $\Omega$ and $\Psi$ are given in lemma~\ref{lem:step generating function asymptotics}.

The integral $\selberg{f(\vec\varphi)}_\varepsilon$ closely resembles
\[
\selberg{f(\vec \varphi)}_H=\int\limits_{-\infty}^\infty\!\!\dots\!\!\int\limits_{-\infty}^\infty
f(\vec\varphi)\left( \prod_{1\le j<m\le k}(\varphi_m-\varphi_j) \right)^2e^{-\sum\limits_{j=1}^k\varphi_j^2/2}
\prod_{j=1}^kd\varphi_j,
\]
a Selberg-like integral with respect to the Hermite weight (hence, the subscript ``$H$'').
Indeed, the change of variables $\varphi_j\sqrt{ n\Lambda}\mapsto\varphi_j$ in $\selberg{f(\vec\varphi)}_\varepsilon$ shows that
\[
\selberg{f(\vec\varphi)}_\varepsilon=(n\Lambda)^{-k^2/2}\selberg{f\left( \frac{\vec\varphi}{\sqrt{n\Lambda}} \right)}_H+O\left( e^{-n\eta} \right),
\qquad n\to\infty,
\]
for some $\eta>0$.
Here, the exponentially small error stems from the fact that $\varepsilon\sqrt n\to\infty$ as $n\to\infty$ and the estimate $\int_y^\infty e^{-x^2/2}dx=O(e^{-y^2/2})$, $y\to\infty$.

These considerations show that \eqref{eq:fixed end point:h1} is asymptotically equal to
\begin{multline*}
	\frac{|\mathcal M|S(1,\dots,1)^n(n\Lambda)^{-k^2/2}}{(2\pi)^kk!}
	\frac{\left( \prod\limits_{1\le j<m\le k}(u_m-u_j) \right)\left( \prod\limits_{1\le j<m\le k}(v_m-v_j) \right)}{\left( \prod\limits_{j=1}^k(j-1)! \right)^2}\\
	\times\Biggl( 
	1
	+\frac{1}{n k\Lambda}\left( \sum_{j=1}^ku_j \right)\left( \sum_{j=1}^kv_j \right)
	-\frac{1}{2n\Lambda}\left( \sum_{j=1}^k(u_j^2+v_j^2) \right)\\
	+\frac{k}{8 n \Lambda^2}\left( k(k^2+2)\Omega+\frac{2k^2+1}{3}\Psi \right)	
	+O\left( n^{-3/2} \right)
	\Biggl)\selberg{1}_H
\end{multline*}
as $n\to\infty$.
Theorem~\ref{thm:lockstep:asymptotics:u->v} is now proved upon recalling (see Appendix~\ref{sec:selberg type integrals}) that
\[
\selberg{1}_H=(2\pi)^{k/2}\prod_{j=1}^kj!.
\]

\section{Free end point: Proofs}
\label{sec:proofs:free}
In order to determine asymptotics for the number of $n$-step configurations with a free end point, we need to sum our integral expression over $v_1<v_2<\dots<v_k$.
More precisely, we need to study
\[
	\sum_{u_1-n-1<v_1<\dots<v_k<u_k+n+1}\det\limits_{1\le j,m\le k}\left( z_j^{-v_m} \right).
\]
Let us first bring the sum to a more convenient form by writing
\begin{multline*}
	\sum_{\substack{v_1,\dots,v_k\\-C<v_1<\dots<v_k<u_k+n+1}}\det\limits_{1\le j,m\le k}\left( z_j^{-v_m} \right)
	\\=
	\det\limits_{1\le j,m\le k}\left( z_j^{-m} \right)
	\left( \prod_{j=1}^kz_j^{k-u_k-n} \right)
	\sum_{\substack{v_1,\dots,v_k\\-C<v_1<\dots<v_k\le u_k+n}}\frac{\det\limits_{1\le j,m\le k}\left( z_j^{u_k+n-v_m} \right)}{\det\limits_{1\le j,m\le k}\left( z_j^{k-m} \right)}.
\end{multline*}
where $-C\le u_1-n-1$.
Now, setting $\lambda_m=u_k+n-v_m-(k-m)$, the expression above is given by
\begin{multline*}
	\det\limits_{1\le j,m\le k}\left( z_j^{-m} \right)
	\left( \prod_{j=1}^kz_j^{k-u_k-n} \right)
	\sum_{\substack{\lambda_1,\dots,\lambda_k\\0\le\lambda_k\le\lambda_{k-1}¸\le\dots\le\lambda_1\le u_k+n+C-k}}
	\frac{\det\limits_{1\le j,m\le k}\left( z_j^{\lambda_m+k-m} \right)}{\det\limits_{1\le j,m\le k}\left( z_j^{k-m} \right)}.
	\\=
	\det\limits_{1\le j,m\le k}\left( z_j^{-m} \right)
	\left( \prod_{j=1}^kz_j^{k-u_k-n} \right)
	\frac{\det\limits_{1\le j,m\le k}\left( z_j^{u_k+n+C-m+1/2}-z_j^{-(k-m+1/2)} \right)}{\det\limits_{1\le j,m\le k}\left( z_j^{k-m+1/2}-z_j^{-(k-m+1/2)} \right)}
\end{multline*}
Letting $C\to\infty$ and noting that $|z_j|<1$, the right hand side converges to
\begin{multline*}
\det\limits_{1\le j,m\le k}\left( z_j^{-m} \right)
\left( \prod_{j=1}^kz_j^{-u_k-n} \right)
\frac{\det\limits_{1\le j,m\le k}\left(-z_j^{m-1/2} \right)}{\det\limits_{1\le j,m\le k}\left( z_j^{k-m+1/2}-z_j^{-(k-m+1/2)} \right)}
\\=
(-1)^{\binom{k}{2}+k}
\det\limits_{1\le j,m\le k}\left( z_j^{-m} \right)
\left( \prod_{j=1}^kz_j^{-u_k-n} \right)
\frac{\det\limits_{1\le j,m\le k}\left(z_j^{m-1/2} \right)}{\det\limits_{1\le j,m\le k}\left( z_j^{m-1/2}-z_j^{-m+1/2} \right)}
\end{multline*}
Now,
%by Lemma\dots,
the right hand side is equal to
\begin{multline*}
(-1)^{\binom{k}{2}+k}
\left( \prod_{j=1}^kz_j^{-u_k-n} \right)
\frac{\det\limits_{1\le j,m\le k}\left( z_j^{k-m} \right)}{\left( \prod\limits_{1\le j<m\le k}(z_jz_m-1) \right)\left( \prod\limits_{j=1}^k(z_j-1) \right)}
\\=
\left( \prod_{j=1}^kz_j^{-u_k-n} \right)
\frac{(-1)^k\det\limits_{1\le j,m\le k}\left( z_j^{m-1} \right)}{\left( \prod\limits_{1\le j<m\le k}(z_jz_m-1) \right)\left( \prod\limits_{j=1}^k(z_j-1) \right)}.
\end{multline*}

For
\[ P_n^+(\vec u)=\sum_{u_1-n-1<v_1<\dots<v_k<u_k+n+1}P_n^+(\vec u\to\vec v) \]
we therefore have the expression
\begin{multline*}
P_n^+(\vec u)
=
\frac{1}{(2\pi i)^kk!}
\idotsint\limits_{|z_1|=\dots=|z_k|=\rho<1}
\frac
	{\det\limits_{1\le j,m\le k}\left( z_j^{u_m-u_k} \right)\det\limits_{1\le j,m\le k}\left( z_j^{m-1} \right)}
	{\left( \prod\limits_{1\le j<m\le k}(z_jz_m-1) \right)\left( \prod\limits_{j=1}^k(1-z_j) \right)}
S(z_1,\dots,z_k)^n
\prod_{j=1}^k \frac{d z_j}{z_j^{n+1}}
\end{multline*}
This last integral is not directly amenable to asymptotic analysis by means of saddle point techniques.
Instead, the integral representation above has first to be translated into a $\lfloor k/2\rfloor$-fold integral, the result of which is summarised in the following lemma.
Its proof is deferred to appendix~\ref{sec:free_end_point_integral_reduction}.
\begin{lemma}
	If $k=2\ell$ is even, then we have
	\begin{multline*}
	P_n^+(\vec u)=\frac{\prod\limits_{j=0}^{\ell-1}(2\ell-2j-1)}{(2\pi i)^{\ell}k!}
	\idotsint\limits_{|w_1|=\dots=|w_\ell|=1}
	\det_{1\le j,m\le k}\left(
	\begin{array}{c@{\hspace*{1cm}}l}
		w_r^{u_m-u_k} & \textrm{if $j=2r-1$} \\
		w_r^{u_k-u_m} & \textrm{if $j=2r$}
	\end{array}
	\right) \\
	\times S\left(w_1,w_1^{-1},\dots,w_\ell,w_\ell^{-1}\right)^n
	\prod_{j=1}^\ell\frac{1+w_j}{1-w_j}\frac{d w_j}{w_j}.
\end{multline*}

	If $k=2\ell+1$ is odd, then we have
	\begin{multline*}
	P_n^+(\vec u)=\frac{\prod\limits_{j=0}^{\ell}(2\ell-2j+1)}{(2\pi i)^{\ell}k!}
	\idotsint\limits_{|w_1|=\dots=|w_\ell|=1}
	\det_{1\le j,m\le k}\left(
	\begin{array}{c@{\hspace*{1cm}}l}
		w_r^{u_m-u_k} & \textrm{if $j=2r-1<k$} \\
		w_r^{u_k-u_m} & \textrm{if $j=2r<k$} \\
		1 & \textrm{if $j=k$}
	\end{array}
	\right) \\
	\times S\left(w_1,w_1^{-1},\dots,w_\ell,w_\ell^{-1},1\right)^n
	\prod_{j=1}^\ell\frac{1+w_j}{1-w_j}\frac{d w_j}{w_j}.
\end{multline*}
	\label{lem:free_end_point_integral_reduction}
\end{lemma}
We proceed with our asymptotic analysis seperately considering the two cases $k$ even and $k$ odd.
\subsection{Asymptotics for $k=2\ell$ even}
Lemma~\ref{lem:free_end_point_integral_reduction} and the substitution $w_r=e^{i\varphi_r}$, $1\le r\le\ell$ yield
\begin{multline*}
P_n^+(\vec u)
=\frac{\prod\limits_{j=0}^{\ell-1}(2\ell-2j-1)}{(2\pi)^{\ell}k!}
\int\limits_{-\pi}^{\pi}\!\!\dots\!\!\int\limits_{-\pi}^\pi
	\det_{1\le j,m\le k}\left(
	\begin{array}{c@{\hspace*{1cm}}l}
		e^{i(u_m-u_k)\varphi_r} & \textrm{if $j=2r-1$} \\
		e^{-i(u_m-u_k)\varphi_r} & \textrm{if $j=2r$}
	\end{array}
	\right) \\
	\times S\left(e^{i\varphi_1},e^{-i\varphi_1},\dots,e^{i\varphi_\ell},e^{-i\varphi_\ell}\right)^n
	\prod_{j=1}^\ell\frac{1+e^{i\varphi_j}}{1-e^{-i\varphi_j}}d\varphi_j.
\end{multline*}
(Recall that there is in fact no pole at points of the form $\varphi_j=0$ as it is cancelled by the zero of the determinant at this point.)
Asymtptotics as $n\to\infty$ of this integral can be established very much like for the quantity $P_n^+(\vec u\to\vec v)$ before by means of the saddle point approach.

Let $\mathcal M$ denote the set of maximal points of
\[
(\varphi_1,\dots,\varphi_\ell)\mapsto|S\left(e^{i\varphi_1},e^{-i\varphi_1},\dots,e^{i\varphi_\ell},e^{-i\varphi_\ell}\right)|
\]
on the set $(-\pi,\pi]^\ell$.

Now, the asymptotically dominant part of $P_n^+(\vec u)$ is captured by
\begin{multline}
	\frac{\Gamma\left( \ell+\frac{1}{2} \right)}{\pi^{\ell+1/2}k!}
\sum_{\hat{\vec\varphi}\in\mathcal M}\idotsint\limits_{\mathcal U_\varepsilon(\hat{\vec\varphi})}
	\det_{1\le j,m\le k}\left(
	\begin{array}{c@{\hspace*{1cm}}l}
		e^{i(u_m-u_k)\varphi_r} & \textrm{if $j=2r-1$} \\
		e^{-i(u_m-u_k)\varphi_r} & \textrm{if $j=2r$}
	\end{array}
	\right) \\
	\times S\left(e^{i\varphi_1},e^{-i\varphi_1},\dots,e^{i\varphi_\ell},e^{-i\varphi_\ell}\right)^n
	\prod_{j=1}^\ell\frac{1+e^{i\varphi_j}}{1-e^{-i\varphi_j}}d\varphi_j,
	\label{eq:dominant part:free end point}
\end{multline}
where
\[
\mathcal U_\varepsilon(\hat{\vec\varphi})=\left\{ \vec\varphi\in\R^k\ :\ |\hat{\vec\varphi}-\vec\varphi|_\infty<\varepsilon \right\},
\qquad \hat{\vec\varphi}=(\hat\varphi_1,\dots,\hat\varphi_k)\in\mathcal M
\]
and $\varepsilon=n^{-5/12}$.
The complementary part of the integral is seen to be exponentially small compared to \eqref{eq:dominant part:free end point} by the very same arguments as given for the fixed end point case, and will therefore not be repeated here.

Also, by the same arguments that proved the validity of equation~\eqref{eq:fixed endpoint:det factorisation} we deduce that
\begin{multline}
	\det_{1\le j,m\le k}\left(
	\begin{array}{c@{\hspace*{1cm}}l}
		e^{i(u_m-u_k)(\hat\varphi_r+\varphi_r)} & \textrm{if $j=2r-1$} \\
		e^{-i(u_m-u_k)(\hat\varphi_r+\varphi_r)} & \textrm{if $j=2r$}
	\end{array}
	\right)
	S\left(e^{\pm i(\hat\varphi_1+\varphi_1)},\dots,e^{\pm i(\hat\varphi_\ell+\varphi_\ell)}\right)^n \\
	=\det_{1\le j,m\le k}\left(
	\begin{array}{c@{\hspace*{1cm}}l}
		e^{i(u_m-u_k)\varphi_r} & \textrm{if $j=2r-1$} \\
		e^{-i(u_m-u_k)\varphi_r} & \textrm{if $j=2r$}
	\end{array}
	\right)
	S\left(e^{\pm i\varphi_1},\dots,e^{\pm i\varphi_\ell}\right)^n.
	\label{eq:free endpoint:det factorisation}
\end{multline}
However, contrary to the fixed end point case, not all points in $\mathcal M$ do contribute the same value to \eqref{eq:dominant part:free end point}.
The reason for this is the factor $\prod_{j=1}^\ell\frac{1+e^{i\varphi_j}}{1-e^{-i\varphi_j}}$.
Let us therefore partition $\mathcal M$ into sets of the form
\[ \mathcal M_a=\left\{ (\hat\varphi_1,\dots,\hat\varphi_\ell)\in\mathcal M\ :\ \sum_{j=1}^\ell\hat\varphi_j=a \pi \right\}, \]
i.e., $\mathcal M=\bigcup\limits_{0\le a\le \ell}\mathcal M_a$.
(Note that the definition of $\mathcal M_a$ does make sense since $\mathcal M\subseteq\left\{ 0,\pi \right\}^\ell$.)
The symmetry of the integrand in \eqref{eq:dominant part:free end point} implies that for each $a$ all points in $\mathcal M_a$ contribute the same value to \eqref{eq:dominant part:free end point}.
Consequently, it suffices to consider for $a\in\{0,1,\dots,\ell\}$ the point $\hat\varphi=(\hat\varphi_1,\dots,\hat\varphi_\ell)\in\mathcal M_a$ such that
\[ \pi=\hat\varphi_1=\cdots=\hat\varphi_a>\hat\varphi_{a+1}=\cdots=\hat\varphi_\ell=0. \]
(We choose $a=0$ and $a=\ell$ for the all $0$'s vector and the all $\pi$'s vector, respectively.)
By the above considerations, the contribution of $\hat\varphi$ to \eqref{eq:dominant part:free end point}  is given by
\begin{multline}
\frac{\Gamma\left( \ell+\frac{1}{2} \right)}{\pi^{\ell+1/2}k!}
\int\limits_{-\varepsilon}^\varepsilon\!\!\cdots\!\!\int\limits_{-\varepsilon}^\varepsilon
	\det_{1\le j,m\le k}\left(
	\begin{array}{c@{\hspace*{1cm}}l}
		e^{i(u_m-u_k)\varphi_r} & \textrm{if $j=2r-1$} \\
		e^{-i(u_m-u_k)\varphi_r} & \textrm{if $j=2r$} \\
	\end{array}
	\right)\\
\times
	S\left(e^{\pm i\varphi_1},\dots,e^{\pm i\varphi_\ell}\right)^n
%	\times\left( \prod_{j=1}^\ell\cos(\varphi_j) \right)^{2n}
	\left(\prod_{j=1}^a\frac{1-e^{i\varphi_j}}{1+e^{i\varphi_j}}\right)
	\left(\prod_{j=a+1}^\ell\frac{1+e^{i\varphi_j}}{1-e^{i\varphi_j}}\right)
	\prod_{j=1}^\ell d\varphi_j,
	\label{eq:free:contribution of arbitrary saddle point}
\end{multline}
where $\varepsilon=\varepsilon(n)=n^{-5/12}$.
Asymptotics for this integral can now be established by replacing the integrand with (sufficiently accurate) Taylor series approximations.
By routine calculations we see that
\begin{align*}
	\prod_{j=a+1}^{\ell}\frac{1+e^{i\varphi_j}}{1-e^{i\varphi_j}}
	&=\frac{(2i)^{\ell-a}}{\prod\limits_{j=a+1}^\ell\varphi_j}\left( 1-\frac{1}{12}\sum_{j=a+1}^\ell\varphi_j^2 +O\left( \max_{a+1\le j\le\ell}|\varphi_j|^4 \right)\right) \\
	\prod_{j=1}^{a}\frac{1-e^{i\varphi_j}}{1+e^{i\varphi_j}}
	&=(2i)^{-a}\left(\prod\limits_{j=1}^a\varphi_j\right)\left( 1+\frac{1}{12}\sum_{j=1}^a\varphi_j^2 +O\left( \max_{1\le j\le a}|\varphi_j|^4 \right)\right).
\end{align*}
A special case of example~\ref{ex:det expansion} is the expansion
\begin{multline*}
\det_{1\le j,m\le k}\left(
	\begin{array}{c@{\hspace*{1cm}}l}
		e^{i(u_m-u_k)\varphi_r} & \textrm{if $j=2r-1$} \\
		e^{-i(u_m-u_k)\varphi_r} & \textrm{if $j=2r$} \\
	\end{array}
	\right) \\
	=
	(-2)^\ell\left( \prod_{j=1}^{k/2}\varphi_j \right)
	\left( \prod_{1\le j<m\le\ell}(\varphi_m^2-\varphi_j^2) \right)^2
	\frac{i^{\binom{k}{2}}}{\prod\limits_{j=1}^{k-1}j!}
	\left(\prod_{1\le j<m\le k}(u_m-u_j)\right) \\
	\times
	\left( 1+ \frac{(k-2)!}{(k+1)!}\left( \sum_{j=1}^\ell\varphi_j^2 \right)\left( \sum_{1\le j<m\le k}(u_m-u_j)^2 \right) + O\left( \max_j|\varphi_j|^3 \right) \right)
\end{multline*}
as $(\varphi_1,\dots,\varphi_\ell)\to(0,\dots,0)$.
Finally, lemma~\ref{lem:step generating function asymptotics} yields
\[
	\log\left| S\left( e^{\pm i\varphi_1},\dots,e^{\pm i\varphi_\ell} \right)\right|=
	\log S(\vec 1)-\Lambda\sum_{j=1}^\ell\varphi_j^2+\frac{\Omega}{2}\left( \sum_{j=1}^\ell\varphi_j^2 \right)^2+\frac{\Psi}{12}\sum_{j=1}^\ell\varphi_j^4+O\left( \max_j|\varphi_j|^6 \right)
\]
as $(\varphi_1,\dots,\varphi_\ell)\to(0,\dots,0)$.

Let us now temporarily denote by $\selberg{f(\varphi_1,\dots,\varphi_\ell)}_\varepsilon$ the integral
\[
\selberg{f(\varphi_1,\dots,\varphi_\ell)}_\varepsilon =
\int\limits_{-\varepsilon}^\varepsilon\!\!\cdots\!\!\int\limits_{-\varepsilon}^\varepsilon
\left( \prod_{1\le j<m\le\ell}(\varphi_m^2-\varphi_j^2) \right)^2f(\varphi_1,\dots,\varphi_\ell)
\prod_{j=1}^\ell e^{-n\Lambda\varphi_j^2}d\varphi_j.
\]
Now, by the above considerations, the integral~\eqref{eq:free:contribution of arbitrary saddle point} is seen to be asymptotically equal to
\[
\mathrm{const}\ \cdot\ S(\vec 1)^n\selberg{\left( \prod_{j=1}^a\varphi_j \right)^2}_\varepsilon,\qquad n\to\infty.
\]
In general, the function $f(\varphi_1,\dots,\varphi_\ell)$ will always be of the form $f(\varphi_1,\dots,\varphi_\ell)=g(\varphi_1^2,\dots,\varphi_\ell^2)$ for some homogenous polynomial $g$.
Therefore, the integral $\selberg{f(\varphi_1,\dots,\varphi_\ell)}_\varepsilon$ may be related to
\[
\selberg{g(\varphi_1,\dots,\varphi_\ell)}_{L;\alpha}=
\int_0^\infty\!\!\dots\!\!\int_0^\infty\left( \prod_{1\le j<m\le\ell}(\varphi_m-\varphi_j) \right)^2g(\varphi_1,\dots,\varphi_\ell)\prod_{j=1}^\ell\varphi_j^{\alpha-1}e^{-\varphi_j}d\varphi_j.
\]
(The informed reader immediately recognises this integral as a Selberg type integral corresponding to the Laguerre weight with parameter $\alpha$.
This also explains the subscript ``$L;\alpha$''.)
Indeed, recalling that $\varepsilon\sqrt n\to\infty$ as $n\to\infty$, folding the integral defining $\selberg{\cdot}_\varepsilon$ followed by the substitution $\varphi_j\sqrt{n\Lambda}\mapsto\varphi_j$  shows that
\[
\selberg{f(\varphi_1,\dots,\varphi_\ell)}_\varepsilon
=
(n\Lambda)^{-\frac{1}{2}\binom{2\ell}{2}-\deg g}\selberg{g(\varphi_1,\dots,\varphi_\ell)}_{L;\frac{1}{2}}+O\left( e^{-n \eta} \right),
\qquad n\to\infty,
\]
for some $\eta>0$, where $\deg g$ denotes the degree of the homogenous polynomial $g$.
Therefore the contribution of the saddle point $\hat\varphi$ to \eqref{eq:free:contribution of arbitrary saddle point} is seen to be
\[
\mathrm{const}\ \cdot\ S(\vec 1)^nn^{-\frac{1}{2}\binom{2\ell}{2}-a},\qquad n\to\infty,
\]
so that the asymptotic significance of the saddle point $\hat\varphi$ decreases for increasing $a$.
In particular, in order to determine the first and second order asymptotics of $P_n^+(\vec u)$, we only need to determine the contributions of $\mathcal M_0$ and $\mathcal M_1$.
The reader may observe at this point that in the random turns model (i.e., $S=\Srt$) there is no contributing saddle point with $a=1$, i.e., $\mathcal M_1=\emptyset$ (which, in fact is true only for $k\ge 2$).
On the other hand, in case of the lock step model the set $\mathcal M_1$ has cardinality $\ell$.

Let us now determine the precise asymptotic contributions for saddle points in $\mathcal M_1$ (if there are any).
More carefully going through the argument above, the contribution of $\mathcal M_1$ to \eqref{eq:dominant part:free end point} -- assuming that $\mathcal M_1\neq\emptyset$ -- is seen to be equal to
\begin{multline}
	4^{\ell}\Gamma\left( \ell+\frac{1}{2} \right)\frac{\prod\limits_{1\le j<m\le k}(u_m-u_j)}{\pi^{\ell+1/2}\prod\limits_{j=1}^{k}j!}
S(\vec 1)^n (n\Lambda)^{-\frac{1}{2}\binom{k}{2}}
\left( -\frac{1}{4n\Lambda}\selberg{\sum_{j=1}^{\ell}\varphi_j}_{L;\frac{1}{2}} +O\left( n^{-5/3} \right)\right)
\label{eq:free:even:contribution a=1}
\end{multline}
as $n\to\infty$.

For the contribution of the saddle point $\hat\varphi=(0,\dots,0)$ (i.e., $a=0$), we need to repeat this argument even more carefully, taking into account the second order contributions.
We find that
\begin{multline}
	4^\ell\Gamma\left( \ell+\frac{1}{2} \right)\frac{\prod\limits_{1\le j<m\le k}(u_m-u_j)}{\pi^{\ell+1/2}\prod\limits_{j=1}^{k}j!}
	S(\vec 1)^n (n\Lambda)^{-\frac{1}{2}\binom{k}{2}} \\
	\times\Bigg( \selberg{1}_{L;\frac{1}{2}}+\frac{1}{n\Lambda}\left( -\frac{1}{12}+\frac{(k-2)!}{(k+1)!}\sum_{1\le j<m\le k}(u_m-u_j)^2 \right)\selberg{\sum_{j=1}^\ell\varphi_j}_{L;\frac{1}{2}} \\
	+\frac{\Omega}{2n\Lambda^2}\selberg{\left( \sum_{j=1}^\ell\varphi_j \right)^2}_{L;\frac{1}{2}}
	+\frac{\Psi}{12n\Lambda^2}\selberg{\sum_{j=1}^\ell\varphi_j^2}_{L;\frac{1}{2}}+O\left( n^{-3/2} \right)\Bigg)
\label{eq:free:even:contribution a=0}
\end{multline}
as $n\to\infty$.

\subsection{Asymptotics for $k=2\ell+1$ odd}
Asymptotics for $k=2\ell+1$ can be determined exactly as in the case $k=2\ell$ we just discussed and we will therefore remain rather brief, focussing on the differences.
In fact, it turns out that for $k=2\ell+1$ odd, the contributions of the saddle points for $a=0,1$ formally look almost identical to equations \eqref{eq:free:even:contribution a=1} and \eqref{eq:free:even:contribution a=0}
except for three differences: the ``$\ell$'' appearing in these formulas should now be interpreted as $\ell=\floor{k/2}$, the factor $\Gamma\left( \ell+\frac{1}{2} \right)$ turns into $\Gamma\left( \ell+\frac{3}{2} \right)$ and the quantities $\selberg{\cdot}_{L;\frac{1}{2}}$ have to be replaced by $\selberg{\cdot}_{L;\frac{3}{2}}$.
The reason for this is the Taylor expansion
\begin{multline*}
	\det_{1\le j,m\le k}\left(
	\begin{array}{c@{\hspace*{1cm}}l}
		e^{i(u_m-u_k)\varphi_r} & \textrm{if $j=2r-1<k$} \\
		e^{-i(u_m-u_k)\varphi_r} & \textrm{if $j=2r<k$} \\
		1 & \textrm{if $j=k$}
	\end{array}
	\right) \\
= 
	(-2)^\ell\left( \prod_{j=1}^\ell\varphi_j^3 \right)
	\left( \prod_{1\le j<m\le\ell}(\varphi_m^2-\varphi_j^2) \right)^2
	\frac{i^{\binom{2\ell}{2}}}{\prod\limits_{j=1}^{k-1}j!}
	\left(\prod_{1\le j<m\le k}(u_m-u_j)\right) \\
	\times
	\left( 1+ \frac{(k-2)!}{(k+1)!}\left( \sum_{j=1}^\ell\varphi_j^2 \right)\left( \sum_{1\le j<m\le k}(u_m-u_j)^2 \right) + O\left( \max_j|\varphi_j|^3 \right) \right)
\end{multline*}
as $(\varphi_1,\dots,\varphi_\ell)\to(0,\dots,0)$ which follows from example~\ref{ex:det expansion}.
The reader should now compare the factor $\prod_{j=1}^\ell\varphi_j^3$ in this formula with the factor $\prod_{j=1}^\ell\varphi_j$ of the expansion of the analogoue of the last subsection.
It is precisely this term that introduces an additional factor of $\prod_{j=1}^\ell\varphi_j^2$ inside the integral $\selberg{\cdot}_\varepsilon$ defined in the last section, which entails the parameter change from $\frac{1}{2}$ to $\frac{3}{2}$.

\subsection{Final form of the asymptotics and proofs of theorems \ref{thm:lockstep:asymptotics:u->} and \ref{thm:randomturns:asymptotics:u->}}
For the sake of convenience, we summarise the results of the last two subsections.
Let $k\ge 2$ be an arbitrary integer and set
\[ 
\ell=\floor{k/2}
\qquad\textrm{and}\qquad
\alpha=\frac{1}{2}+k-2\ell.
\]
The integrals $\selberg{\cdot}_{L;\alpha}$ are always understood as $\ell$-fold integrals.
Then, the contribution of the saddle point $\hat\varphi=(0,\dots,0)$ (i.e., of $\mathcal M_0$) is given by
\begin{multline}
	4^\ell\Gamma\left( \ell+\alpha \right)\frac{\prod\limits_{1\le j<m\le k}(u_m-u_j)}{\pi^{\ell+1/2}\prod\limits_{j=1}^{k}j!}
	S(\vec 1)^n (n\Lambda)^{-\frac{1}{2}\binom{k}{2}}
	\Bigg(
		1
		+\frac{\ell(\ell+\alpha-1)}{n\Lambda}\Xi
		+O\left( n^{-3/2}\right)
	\Bigg)
	\selberg{1}_{L;\alpha}
\label{eq:free:contribution a=0}
\end{multline}
as $n\to\infty$, where
\[
\Xi =-\frac{1}{12}
		+\frac{\Omega(1+\ell(\ell+\alpha-1))}{2\Lambda}
		+\frac{\Psi(2\ell+\alpha-1)}{12\Lambda}
		+\frac{(k-2)!}{(k+1)!}\sum_{1\le j<m\le k}(u_m-u_j)^2.
\]
If $\mathcal M_1\neq\emptyset$, then its contribution to the asymptotics of $P_n^+(\vec u)$ is equal to
\begin{equation}
	4^\ell\Gamma\left( \ell+\alpha \right)\frac{\prod\limits_{1\le j<m\le k}(u_m-u_j)}{\pi^{\ell+1/2}\prod\limits_{j=1}^{k}j!}
S(\vec 1)^n (n\Lambda)^{-\frac{1}{2}\binom{k}{2}}
\left( -\frac{\ell(\ell+\alpha-1)}{4n\Lambda}+O\left( n^{-5/3} \right)\right)\selberg{1}_{L;\alpha}
\label{eq:free:contribution a=1}
\end{equation}
as $n\to\infty$.
The final form is then obtained by noting that
\[ \Gamma(\ell+\alpha)\selberg{1}_{L;\alpha}=\Gamma(\ell+\alpha)\prod_{j=0}^{\ell-1} (j+1)!\Gamma\left( j+\alpha \right)=\prod_{j=0}^\ell j!\Gamma(j+\alpha). \]

Now, in the lock step model, i.e., $S=\Sls$, we know that $\mathcal M_1\neq\emptyset$ so that in this case, the asymptotics of $P_n^+(\vec u)$ is given by the sum of equations \eqref{eq:free:contribution a=0} and \eqref{eq:free:contribution a=1}. This proves theorem~\ref{thm:lockstep:asymptotics:u->} .
On the other hand, in case of the random turns model we know that for $k\ge 2$ we have $\mathcal M_1=\emptyset$ so that in this case the asymptotics of $P_n^+(\vec u)$ is given by equation \eqref{eq:free:contribution a=0}, which completes the proof of theorem~\ref{thm:randomturns:asymptotics:u->}

\appendix
\section{Proof of lemma~\ref{lem:free_end_point_integral_reduction}}
\label{sec:free_end_point_integral_reduction}
Recall that in lemma~\ref{lem:free_end_point_integral_reduction}, we claimed that the quantity (the number of walks of length $n$ starting in $\vec u$)
\begin{multline*}
P_n^+(\vec u)
=
\frac{1}{(2\pi i)^kk!}
\idotsint\limits_{|z_1|=\dots=|z_k|=\rho<1}
\frac
	{\det\limits_{1\le j,m\le k}\left( z_j^{u_m-u_k} \right)\det\limits_{1\le j,m\le k}\left( z_j^{m-1} \right)}
	{\left( \prod\limits_{1\le j<m\le k}(z_jz_m-1) \right)\left( \prod\limits_{j=1}^k(1-z_j) \right)}
S(z_1,\dots,z_k)^n
\prod_{j=1}^k \frac{d z_j}{z_j^{n+1}}
\end{multline*}
can be represented as a $\floor{k/2}$-fold integral (the precise form being dependent on the parity of $k$).

The main idea of the proof is the following.
In the integral above, we successively ``push the contours of integration to infinity'' taking into account the residues we encounter.
At each step, we have to rewrite the expression in a way that allows us repeat this procedure.
This rewriting essentially consists of properly interchanging the order of integration (justified by Fubini's Theorem) and  relabelling and reordering of the integration variables involved.

\begin{notation*}
	For a vector $\vec z=(z_1,\dots,z_k)$ we denote by $\vec z^{(r)}=(z_1,\dots,z_{k-r})$ the vector obtained from $\vec z$ by removing the last $r$ components.
	For $\vec z^{(1)}$ and $\vec z^{(2)}$ we also write $\vec z'$ and $\vec z''$, respectively.
	Furthermore, by a slight abuse of notation, we identify $\vec z$ and $(\vec z^{(r)},z_{k-r+1},\dots,z_k)$, so that, e.g., $F(\vec z',z_k)=F(\vec z)=F(z_1,\dots,z_k)$.

	Additionally, we will abbreviate $F(\dots,z_j,z_j^{-1},\dots)$ by $F(\dots,z_j^\pm,\dots)$, which will become quite handy, too.
\end{notation*}

For the sake of convenience, we define
\begin{align*}
	G_\ell(z_1,\dots,z_{\ell}) &= 
	\det\limits_{1\le j,m\le \ell}\left( z_j^{m-1} \right)\left( \prod\limits_{1\le j<m\le \ell}(z_jz_m-1) \right)^{-1}
	\left(\prod_{j=1}^\ell (1-z_j)z_j^{n+1}\right)^{-1}
	\\
	F(z_1,\dots,z_k) &= \det\limits_{1\le j,m\le k}\left( z_j^{u_m-u_k} \right)S\left( z_1,\dots,z_k \right)^n.
\end{align*}
In terms of these quantities, the integral we are interested in (or rather a multiple thereof) can be written as
\begin{equation*}
	I_{k,k}^{(\mathrm{even})}=k! P_n^+(\vec u) =
	\frac{1}{(2\pi i)^k}\idotsint\limits_{0<|z_1|<|z_2|<\dots<|z_k|<1}
	G_k(\vec z)
	F(\vec z)
	\prod_{j=1}^kdz_j.
\end{equation*}
In the following, we will transform this $k$-fold integral into a $\floor{k/2}$-fold integral.
As already mentioned, this is accomplished by ``pushing every other contour to infinity'' (starting with the contour for $z_k$), at each step taking into account the residues we encounter.

In general, we will encounter quantities that are either of the form
\begin{multline*}
	I_{k,k-2r}^{(\mathrm{even})} =
	\frac{1}{(2\pi i)^{k-r}}\idotsint\limits_{\substack{0<|z_1|<\dots<|z_{k-2r}|<1 \\ |w_1|=\dots=|w_r|=\rho}}G_{k-2r}(\vec z^{(2r)})F(\vec z^{(2r)},w_1^\pm,\dots,w_r^\pm)
	\prod_{j=1}^{k-2r}d z_j \prod_{j=1}^r \frac{1+w_j}{1-w_j}\frac{dw_j}{w_j}
\end{multline*}
or of the form
\begin{multline*}
	I_{k,k-2r-1}^{(\mathrm{odd})} =
	\frac{(-1)^{k-1}}{(2\pi i)^{k-r-1}}\idotsint\limits_{\substack{0<|z_1|<\dots<|z_{k-2r-1}|<1 \\ |w_1|=\dots=|w_r|=\rho}}G_{k-2r-1}(\vec z^{(2r+1)})F(\vec z^{(2r+1)},w_1^\pm,\dots,w_r^\pm,1)
	\\ \times\prod_{j=1}^{k-2r-1}d z_j \prod_{j=1}^r\frac{1+w_j}{1-w_j}\frac{dw_j}{w_j}.
\end{multline*}
The reader should be aware that in fact the integrands have no pole at $w_j=1$ since by lemma~\ref{lem:F_1_1} below we have $\left.F(\dots,w_j^\pm,\dots)\right|_{w_j=1}=0$.

With the above notation at hand, lemma~\ref{lem:free_end_point_integral_reduction} may be rephrased as
\begin{lemma*}
	For $k\ge 0$, we have
	\begin{equation*}
		I_{k,k}^{(\mathrm{even})} =
		M_{k}\times
		\begin{cases}
			I_{k,0}^{(\mathrm{even})} & \textrm{if $k$ is even} \\
			I_{k,0}^{(\mathrm{odd})} & \textrm{if $k$ is odd,} 
		\end{cases}
	\end{equation*}
	where
	\begin{equation*}
		M_{2\ell-1} = M_{2\ell} = \prod\limits_{j=0}^{\ell-1}(2\ell-2j-1),\qquad \ell\ge 1.
	\end{equation*}
\end{lemma*}
This lemma in turn follows directly from the following recursion.
\begin{lemma*}
	We have the recursion
	\begin{align*}
		I_{k,k-2r}^{(\mathrm{even})} &= (-1)^{k-1}I_{k,k-2r-1}^{(\mathrm{odd})}+(k-2r-1)I_{k,k-2r-2}^{(\mathrm{even})} \\
		I_{k,k-2r-1}^{(\mathrm{odd})} &= (k-2r-2)I_{k,k-2r-3}^{(\mathrm{odd})}
	\end{align*}
\end{lemma*}
It therefore suffices to prove the recursion claimed in this last lemma.

\subsection{The reduction step}
In this subsection, we start our procedure by ``pushing the contour of integration for $z_k$ in $I_{k,k}^{(\mathrm{even})}$ to infinity''.

Let $\Res_{z=\zeta}f(z)$ denote the residue of $f(z)$ at $z=\zeta$.
Routine calculations show that we have the following evaluations.
\begin{lemma}
	We have
	\begin{align*}
		\Res\limits_{z_\ell=1} G_{\ell}(z_1,\dots,z_\ell) &= -(-1)^{k-1}G_{\ell-1}(z_1,\dots,z_{\ell-1}),
		\\
		\Res\limits_{z_\ell=1/z_{j}} G_{\ell}(z_1,\dots,z_\ell) &= -G_{\ell-2}(z_1,\dots,z_{j-1},z_{j+1},\dots,z_{\ell-1})\frac{1+z_{j}}{1-z_{j}}z_{j}^{-1}
	\end{align*}
	for $1\le\ell\le k$ and $1\le j<\ell\le k$, respectively.
	\label{lem:G_residues}
\end{lemma}

\begin{lemma}
	For $1\le j\le\ell\le k$ we have the asymptotics
	\begin{align*}
		G_{\ell}(z_1,\dots,z_\ell) &= O\left( z_j^{-n-2} \right),
		\\
		F(z_1,\dots,z_k) &= O\left( z_j^{n} \right)
	\end{align*}
	as $z_j\to\infty$.
	\label{lem:F_G_asymptotics}
\end{lemma}

\begin{lemma}
	For $\vec z=(z_1,z_2,\dots,z_k)$ and $1\le\ell\le k$, the quantity $G_{k-\ell}(\vec z^{(\ell)})F(\vec z)$ is symmetric with respect to $z_1,\dots,z_{k-\ell}$.
	\label{lem:F_G_symmetry}
\end{lemma}

Let us now turn our attention towards the $k$-fold integral $I_{k,k}^{(\mathrm{even})}$.
For the innermost integral we have by Cauchy's theorem
\begin{multline*}
	\frac{1}{2\pi i}\int\limits_{|z_k|<1}G_k(\vec z)F(\vec z) dz_k
=
-F(\vec z',1)\Res_{z_k=1}\left( G_k(\vec z) \right)
-\sum_{j=1}^{k-1}F(\vec z',z_j^{-1})\Res_{z_k=1/z_{j}}\left( G_k(\vec z) \right) \\
+\frac{1}{2\pi i}\int\limits_{|z_k|=R>1/|z_1|}G_k(\vec z)F(\vec z) dz_k.
\end{multline*}
The reader should now observe that all the poles of $G_k(\vec z)F(\vec z)$ lie inside the circle $|z_k|=R$.
Also, by lemma~\ref{lem:F_G_asymptotics} we know that $G_k(\vec z',z_k)F(\vec z',z_k)=O(z_k^{-2})$ as $z_k\to\infty$.
Consequently, by Cauchy's theorem and trivial bounds, we may deduce that
\[
\frac{1}{2\pi i}\int\limits_{|z_k|=R>1/|z_1|}G_k(\vec z)F(\vec z) dz_k=0.
\]
We therefore have by lemma~\ref{lem:G_residues} and relabelling of integration variables (justified by lemma~\ref{lem:F_G_symmetry}) for $k\ge 2$ the relation
\[
I_{k,k}^{(\mathrm{even})}=(-1)^{k-1}I_{k,k-1}^{(\mathrm{odd})}+(k-1)I_{k,k-2}^{(\mathrm{even})}.
\]

\subsection{The recursion}
The arguments given in the last subsection (properly modified) readily give the relation
\[ 
I_{k,k-2r}^{(\mathrm{even})}=(-1)^{k-1}I_{k,k-2r-1}^{(\mathrm{odd})}+(k-2r-1)I_{k,k-2r-2}^{(\mathrm{even})}
\]
for $0\le 2r<k-1$.
This is seen by noting that all the poles of the integrands are in fact poles of $G_{\ell}(z_1,\dots,z_\ell)$ and therefore do not involve variables other than $z_1,\dots,z_\ell$.
We refrain from giving details.

A recursion relation for $I_{k,k-2r-1}^{(\mathrm{odd})}$, $0\le 2r<k$, can be obtained in a pretty similar fashion.
However, there is one important difference that is based on the following simple observation.
\begin{lemma}
	For $z_1,\dots,z_{k-2}\neq 0$ we have
	\begin{equation*}
		F(z_1,\dots,z_{k-2},1,1) = 0
	\end{equation*}
	so that, by symmetry, we have $F(\dots,1,\dots,1,\dots)=0$.
	\label{lem:F_1_1}
\end{lemma}
The above claim readily follows from the definition of $F(\vec z)$ by noting that in this case the determinant involved is equal to zero.

Now, consider the integrand of $I_{k,k-2r-1}^{(\mathrm{odd})}$.
If we push to infinity the contour of $z_{k-2r-1}$ (as we did in the last subsection) we see that the pole of $G_{k-2r-1}(z_1,\dots,z_{k-2r-1})$ at $z_{k-2r-1}=1$ is cancelled by the zero of $F(z_1,\dots,z_{2k-2r-1},w_1^\pm,\dots,w_r^{\pm},1)$ at this point.
Consequently, we see that for $0<r$ we have
\[ I_{k,k-2r-1}^{(\mathrm{odd})} = (k-2r-2)I_{k,k-2r-3}^{(\mathrm{odd})}. \]
This proves the lemma.

\section{Integral evaluations related to Selberg type integrals}
\label{sec:selberg type integrals}

The asymptotic analysis conducted in the previous chapters required us to evaluate certain multiple integrals of the form
\[
\selberg{f(\vec x)}=\int\limits_\Gamma\!\cdots\!\int\limits_\Gamma f(\vec x)\Phi(\vec x)d\vec x,
\]
where $\vec x = (x_1,\dots,x_k)$, $d\vec x=d x_1\cdots d x_k$ and either
\[
\Phi(\vec x)=\Phi_{L;\alpha}(\vec x)=\left( \prod_{j=1}^k x_j^{\alpha-1} \right)\left( \prod_{1\le j<m\le k}(x_m-x_j) \right)^2e^{-\sum_{j=1}^kx_j}
\qquad\textrm{and}\qquad
\Gamma=[0,\infty)
\]
or
\[
\Phi(\vec x)=\Phi_H(\vec x)=\left( \prod_{1\le j<m\le k}(x_m-x_j) \right)^2e^{-\sum_{j=1}^kx_j^2/2}
\qquad\textrm{and}\qquad
\Gamma=(-\infty,\infty).
\]
The corresponding integrals are denoted by $\selberg{f(\vec x)}_{L;\alpha}$ and $\selberg{f(\vec x)}_H$, respectively.
The subscripts ``$L;\alpha$'' and ``$H$'' are chosen because in random matrix theory, the corresponding integrals a intimately related to the so called \emph{Laguerre ensemble with parameter $\alpha$} and the \emph{Hermite ensemble}, respectively.
In the relevant cases, $f(\vec x)$ always was some symmetric polynomial.

For $f(x)=1$, both integrals are special cases of the well-known Selberg integral, viz
\[
\int\limits_0^1\!\cdots\!\int\limits_0^1\left( \prod_{1\le j<m\le k}\left|t_m-t_j\right|^{2\lambda} \right)\prod_{j=1}^kt_j^{\lambda_1}(1-t_j)^{\lambda_2}d t_j.
\]
More precisely, the integrals $\selberg{1}_L$ and $\selberg{1}_H$ can be obtained from the Selberg integral above by choosing specific values for $\lambda,\lambda_1,\lambda_2$ and taking limits (we refer the reader to \cite[Chapter 17]{MR2129906} for details).

We end this section by providing the reader with a list of relevant integral evaluations for specific choices of $f(\vec x)$.
It should be noted that all these evaluations are well-known, and the literature contains various different proofs for most of them.

\begin{lemma}
	We have
	\begin{equation}
		\selberg{1}_{L;\alpha} = \prod_{j=0}^{k-1}(j+1)!\Gamma(j+\alpha)
		\label{eq:selberg:laguerre:1}
	\end{equation}
	and
	\begin{equation}
\selberg{1}_H = (2\pi)^{k/2}\prod_{j=1}^k j!
	\label{eq:selberg:hermite:1}
	\end{equation}
	\label{lem:selberg:1}
\end{lemma}

\begin{lemma}
	We have the evaluations
	\begin{align*}
		\selberg{\sum_{j=1}^kx_j^2}_H &= k^2\selberg{1}_H & \qquad
		\selberg{\left( \sum_{j=1}^kx_j \right)^2}_H &= k\selberg{1}_H \\ 
		\selberg{\sum_{j=1}^kx_j^4}_H &= k\left( 2k^2+1 \right)\selberg{1}_H & \qquad
		\selberg{\left( \sum_{j=1}^kx_j^2 \right)^2}_H &= k^2(k^2+2)\selberg{1}_H.
	\end{align*}
	\label{lem:selberg hermite evaluations}
\end{lemma}
\begin{proof}
	For the first identity, we consider
	\[ \frac{\partial}{\partial x_m}x_m\Phi_H(\vec x) = \Phi_H(\vec x)\left( 1+2\sum_{j\neq m}\frac{x_m}{x_m-x_j}-x_m^2 \right). \]
	Summing this equation over $m=1,\dots,k$ and integrating over $(-\infty,\infty)^k$ yields
	\[ 0 = k^2\selberg{1}_H-\selberg{\sum_{j=1}^kx_j^2}_H, \]
	which proves the second evaluation.

	The remaining claims are proved in the same fashion.
	Integrating the equation
	\[
	\sum_{m=1}^k\sum_{j=1}^k\frac{\partial}{\partial x_j}x_m\Phi_H(\vec x) = \Phi_H(\vec x)\left( k-\left( \sum_{j=1}^kx_j \right)^2 \right)
	\]
	and noting that, again, the integral over the left hand side is equal to zero, we obtain the second evaluation.
	The third claim can be validated by integrating 
	\[
	\sum_{m=1}^k\frac{\partial}{\partial x_m}x_m^3\Phi_H(\vec x) = \Phi_H(\vec x)\left( 2k\sum_{j=1}^kx_j^2+\left( \sum_{j=1}^kx_j \right)^2-\sum_{j=1}^kx_j^4 \right).
	\]
	Finally, for the fourth claim we integrate the equation
	\[
	\sum_{\ell=1}^k\sum_{m=1}^k\frac{\partial}{\partial x_m}x_mx_\ell^2\Phi_H(\vec x) = \Phi_H(\vec x)\left( (k^2+2)\sum_{j=1}^kx_j^2-\left(\sum_{j=1}^kx_j^2\right)^2 \right).
	\]
\end{proof}
\begin{corollary}
	\begin{equation*}
		\selberg{s_{(1,1)}(\vec x)}_H = -\binom{k}{2}\selberg{1}_H \qquad\qquad% \label{eq:selberg hermite evaluations:s_(1,1)}\\
		\selberg{s_{(2)}(\vec x)}_H = \binom{k+1}{2}\selberg{1}_H% \label{eq:selberg hermite evaluations:s_(2)}
	\end{equation*}
\end{corollary}
This readily follows from lemma~\ref{lem:selberg hermite evaluations} by noting that
\begin{align*}
	\left( \sum_{j=1}^kx_j \right)^2-\sum_{j=1}^kx_j^2 &= 2s_{(1,1)}(x_1,\dots,x_k) \\
	\left( \sum_{j=1}^kx_j \right)^2+\sum_{j=1}^kx_j^2 &= 2s_{(2)}(x_1,\dots,x_k).
\end{align*}

Analogously, we have the following evaluations in the Laguerre case.
\begin{lemma}
	We have the evaluations
	\begin{align*}
		\selberg{\sum_{j=1}^kx_k}_{L;\alpha} &= k(k-1+\alpha)\selberg{1}_{L;\alpha} \\
		\selberg{\sum_{j=1}^kx_k^2}_{L;\alpha} &= (2k+\alpha-1)\selberg{\sum_{j=1}^kx_k}_{L;\alpha} \\
		\selberg{\left(\sum_{j=1}^kx_k\right)^2}_{L;\alpha} &=  \left( 1+k(k+\alpha-1) \right)\selberg{\sum_{j=1}^kx_k}_{L;\alpha}.
	\end{align*}
	\label{lem:selberg laguerre evaluations}
\end{lemma}
\begin{proof}
	The claimed evaluations can, again, be validated with Aomoto's technique.
	We integrate
	\[
	\sum_{j=1}^k\frac{d}{d x_j}x_j\Phi_{L;\alpha}(\vec x) = \Phi_{L;\alpha}(\vec x)\left(k\alpha+\sum_{j=1}^k\sum_{r\neq j}\frac{2x_r}{x_r-x_j}-\sum_{j=1}^kx_j  \right)
	\]
	for the first evaluation,
	\[
	\sum_{j=1}^k\frac{d}{d x_j}x_j^2\Phi_{L;\alpha}(\vec x) = \Phi_{L;\alpha}(\vec x)\left( (\alpha+1)\sum_{j=1}^kx_j+\sum_{j=1}^k\sum_{r\neq j}\frac{2x_r^2}{x_r-x_j}-\sum_{j=1}^kx_j^2  \right)
	\]
	for the second claim and
	\[
	\sum_{j,m=1}^k\frac{d}{d x_j}x_jx_m\Phi_{L;\alpha}(\vec x) = \Phi_{L;\alpha}(\vec x)\left( (1+k(k+\alpha-1))\sum_{j=1}^kx_j-\left(\sum_{j=1}^kx_j\right)^2  \right)
	\]
	for the last one.
\end{proof}

\bibliographystyle{plain}
\bibliography{randomwalks}

\end{document}